\documentclass[12pt]{amsart}
\usepackage{amscd}     
\usepackage{amssymb}
\usepackage{amsmath, amsthm, graphics}
\usepackage{xypic}     
\LaTeXdiagrams        
\usepackage[all]{xy}
\xyoption{2cell} \UseAllTwocells \xyoption{frame} \CompileMatrices
\allowdisplaybreaks[3]
\usepackage{amsfonts}
\usepackage[normalem]{ulem}
\usepackage{tikz}
\usepackage{tikz-cd}
\usepackage{mathtools}
\usepackage{verbatim} 
\usepackage{hyperref}

\usepackage[margin=4cm]{geometry}

\usepackage{graphicx} 

\usepackage{latexsym}
\usepackage{epsfig}
\usepackage{amsfonts}
\usepackage{enumerate}
\usepackage{times,float}
\usepackage{graphics}
\usepackage{hyperref}

\newtheorem{prop}{Proposition}
\newtheorem{lem}[prop]{Lemma}

\newtheorem{thm}[prop]{Theorem}

\newtheorem{conjecture}[prop]{Conjecture}

\theoremstyle{remark}

\theoremstyle{remark}

\numberwithin{equation}{section}

\newcommand{\Mbar}{\overline{\M}}

\newcommand{\com}{\mathbb{C}}

\newcommand{\X}{\mathcal{X}}

\newcommand{\M}{\mathcal{M}}

\newcommand{\T}{\mathsf{T}}

\newcommand{\Hilb}{\mathsf{Hilb}^n(\mathbb{C}^2)}

\newcommand{\hilbnc}{\mathsf{Hilb}^n(\mathbb{C}^2)}

\newcommand{\lv}{\left |}

\newcommand{\rang}{\right\rangle}
\newcommand{\blang}{\big\langle}
\newcommand{\brang}{\big\rangle}

\newcommand{\la}{\langle}
\newcommand{\ra}{\rangle}

\newcommand{\RRR}{\mathsf{R}}

\newcommand{\CC}{\mathbb{C}}

\newcommand{\cO}{\mathcal{O}}

\newcommand{\cI}{\mathcal{I}}

\newcommand{\NL}{\mathsf{NL}}

\def\b1{{\mathbf 1}}

\begin{document}

\setcounter{section}{-1}

\title[Gromov-Witten theory of $\Hilb$ and Noether-Lefschetz theory of $\mathcal{A}_g$]{Gromov-Witten theory of $\Hilb$ and
 Noether-Lefschetz theory of $\mathcal{A}_g$}
\author[Iribar L\'opez]{Aitor Iribar L\'opez}
\address{Department of  Mathematics\\ ETH Z\"urich \\ R\"amistrasse 101 \\ 8092 Z\"urich\\ Switzerland}
\email{aitor.iribarlopez@math.ethz.ch}

\author[Pandharipande]{Rahul Pandharipande}
\address{Department of  Mathematics\\ ETH Z\"urich \\ R\"amistrasse 101 \\ 8092 Z\"urich\\ Switzerland}
\email{rahul@math.ethz.ch}

\author[Tseng]{Hsian-Hua Tseng}
\address{Department of Mathematics\\ Ohio State University\\ 100 Math Tower, 231 West 18th Ave. \\ Columbus,  OH 43210\\ USA}
\email{hhtseng@math.ohio-state.edu}

\date{August 2025}

\begin{abstract}
We calculate the genus 1 Gromov-Witten theory of the Hilbert scheme 
$\Hilb$ of points in the plane. The fundamental 1-point invariant (with a divisor insertion) is calculated using a correspondence with
the families local curve Gromov-Witten theory 
over the moduli space $\overline{\mathcal{M}}_{1,1}$. The answer exactly matches
a parallel calculation related to the Noether-Lefschetz geometry of
the moduli space $\mathcal{A}_g$ of principally polarized abelian varieties.
As a consequence, we prove that the associated cycle classes satisfy a
homomorphism property for the projection operator on $\mathsf{CH}^*(\mathcal{A}_g)$.
The fundamental 1-point invariant determines the full genus 1 Gromov-Witten
theory of $\Hilb$ modulo a nondegeneracy conjecture about the quantum cohomology.
A table of calculations is given.

\end{abstract}

\maketitle

\setcounter{tocdepth}{1}
\tableofcontents

\section{Introduction}
\subsection{Hilbert schemes}
\label{sechs}
The Hilbert scheme $\Hilb$ of $n$ points in the plane $\CC^2$ is a nonsingular, irreducible,
quasi-projective variety of dimension $2n$ parameterizing ideals $\cI\subset \CC[x,y]$ of colength $n$,
$$
\dim_\CC {\CC[x,y]}/{\cI} = n \,. 
$$
An open dense subset of $\Hilb$ parameterizes ideals associated to configurations of $n$ distinct unordered points. The geometry of $\Hilb$ has been studied from many points
of view for several decades now, see
\cite{Fog,goeB,goe,Groj, Nak}. Our perspective here is related
to the interactions of $\Hilb$
with Gromov-Witten theory and the relative Donaldson-Thomas
invariants of threefolds as developed in \cite{bp,MNOP1,MNOP2,MR,MO,op,opdt, op2,PanT,pt,pt-hs}.

The algebraic torus $\T=(\CC^*)^2$ acts diagonally on $\CC^2$ by scaling coordinates,
$$
(z_1,z_2) \cdot (x,y) = (z_1 x, z_2 y)\, .
$$
Let $t_1$ and $t_2$ denote the equivariant parameters corresponding to the weights of the $\T$-action on the tangent space $\text{Tan}_0(\CC^2)$ at the origin of $\CC^2$.

There is a canonically induced $\T$-action on $\Hilb$. The associated $\T$-equivariant cohomology,
$H_{\T}^*(\Hilb,{\mathbb Q})$, admits a natural basis (as a $\mathbb{Q}[t_1,t_2]$-module) called the {\em Nakajima basis}. The Nakajima basis element $\lv \mu \rang$ corresponding to the partition $\mu$ of $n$ is
$$\frac{1}{\Pi_i \mu_i} [V_\mu]$$
where $[V_\mu]$ is (the cohomological dual of)
the class of the subvariety of $\mathsf{Hilb}^{|\mu|}(\CC^2)$
with generic element given by a union of 
schemes of lengths $$\mu_1, \ldots, \mu_{\ell(\mu)}$$ supported
at $\ell(\mu)$ distinct points{\footnote{The points and
parts of $\mu$ are considered
here to be unordered.}} of $\CC^2$. The element $\lv 1^n \rang$ corresponds to the unit 
$$\mathsf{1}\in H^*_\T(\Hilb,{\mathbb Q})\, .$$ 
See \cite{NakB} for a foundational treatment.

The Hilbert scheme carries a tautological rank $n$ vector bundle, 
\begin{equation}\label{xx88}
\cO/\cI\rightarrow \Hilb\, ,
\end{equation}
with fiber
$\CC[x,y]/\cI$ over $[\cI]\in \Hilb$.  
The $\T$-action on $\Hilb$ lifts canonically to the tautological bundle \eqref{xx88}.
Let $$
D = c_1 (\cO/\cI) \in H^2_{\T}(\Hilb, {\mathbb Q})$$
be the $\T$-equivariant first Chern class. A straightforward calculation{\footnote{The $n=0,1$ cases are degenerate: $D=0$ for both.}} for $n\geq 2$ shows 
\begin{equation*}
D = - \lv 2,1^{n-2} \rang \,,
\end{equation*}
see \cite{Lehn}.

\subsection{Quantum cohomology}
The $\T$-equivariant quantum cohomology of $\Hilb$  has been determined in \cite{op}.
The {\em matrix elements} of the $\T$-equivariant quantum product count{\footnote{The count is virtual.}} rational curves meeting three given subvarieties of $\Hilb$. 
The (non-negative) {\em degree} of an effective{\footnote{The $\beta=0$
is considered here effective.}} curve class 
$$\beta\in H_2(\Hilb,{\mathbb Z})$$ is defined by pairing with $D$,
$$d=\int_\beta D\, .$$
Curves of degree $d$ are counted with 
weight $q^d$, where $q$ is the quantum parameter.
The ordinary multiplication in $\T$-equivariant cohomology is recovered by
setting $q=0$.

Let $\mathsf{M}_D^{\mathsf{Hilb}^n(\mathbb{C}^2)}$ be the operator of quantum multiplication{\footnote{Here, the
symbol $\star$ denotes the small quantum product. In Section \ref{sec:recon_multipt_giv}, the large quantum product will also play a role (and will be denoted by $\star_t$).} by the divisor $D$,
$$\mathsf{M}_D^{\mathsf{Hilb}^n(\mathbb{C}^2)}: 
QH^*_\T(\Hilb) \rightarrow 
QH^*_\T(\Hilb)\, , \ \ \ \ \ \mathsf{M}_D^{\mathsf{Hilb}^n(\mathbb{C}^2)}(\gamma) =
D \star \gamma\, .$$
The operator $\mathsf{M}^{\Hilb}_D$ is calculated explicitly in the Nakajima basis for all $\hilbnc$ in \cite{op}. The matrix coefficients of $\mathsf{M}^{\Hilb}_D$ lie in the field of rational functions
in $q$ (with coefficients{\footnote{In the context of quantum cohomology,
the definitions require localization, so we will always consider $H^*_\T(\Hilb)$ and
$QH^*_\T(\Hilb)$ as modules
over $\mathbb{Q}(t_1,t_2)$.}}
in $\mathbb{Q}(t_1,t_2)$) and determine all genus $0$ Gromov-Witten invariants of $\hilbnc$ by \cite[Section 4.2]{op}.

We are interested here in
the $\T$-equivariant Gromov-Witten invariants of $\Hilb$ in genus 1.
Let $\mu^1,\ldots, \mu^r$ be partitions of $n$. Define
\begin{equation}\label{eqn:g1_multi}
\blang \mu^1, \mu^2, \ldots, \mu^r \brang_{1}^{\Hilb}
\, =\,  \sum_{d=0}^\infty 
\blang \mu^1, \mu^2, \ldots, \mu^r \brang_{1,d}^{\Hilb} q^d
\, \in \mathbb{Q}(t_1,t_2)[[q]].
\end{equation}
The series \eqref{eqn:g1_multi} is always a rational{\footnote{There
is a gap in the proof of rationality in the published paper \cite{pt} which
is filled in a \href{https://people.math.ethz.ch/~rahul/HilbC2-2025-August.pdf}{revision}
available on the authors' websites.
The gap does not affect the proof in \cite{pt} of the
triangle of correspondences  \cite[Theorem 2]{pt} in the form of
analytic continuation from $q=0$ to $q=-1$, which is all that is needed here.
}} function \cite{pt} in 
$q$, 
$$\blang \mu^1, \mu^2, \ldots, \mu^r \brang_{1}^{\Hilb}
 \in \mathbb{Q}(t_1,t_2,q)\, .\, $$
The first nontrivial computation in genus 1 appeared in \cite{pt}:
\begin{equation*}
\big\langle \, D\,  \big\rangle_{1}^{\mathsf{Hilb}^2(\mathbb{C}^2)} = -\frac{1}{24}\frac{(t_1+t_2)^2}{t_1t_2} \cdot \frac{q+1}{q-1}\, .
\end{equation*}
 Our goal here is to provide calculations of all of the series 
 \eqref{eqn:g1_multi}
 starting
 with the basic case of $\blang D \brang_{1}^{\Hilb}$.

The genus 0 Gromov-Witten theory of $\hilbnc$ is
a very rich subject with connections to the quantum Calogero-Sutherland many body system \cite{op},
to the crepant resolution conjecture \cite{bg}, and
to quantum groups \cite{MOk}. The question of computing $\blang D \brang_{1}^{\Hilb}$
for all $n$ was posed in \cite{P,pt}.
We will see that the genus 1 Gromov-Witten theory of $\hilbnc$ provides
new connections: the theory is fundamentally linked to the geometry of
the moduli space of abelian varieties. 

There are several natural open questions. What is the structure of the genus 1 Gromov-Witten theory of more general Nakajima quiver varieties?
Is there an exact calculation of the full genus Gromov-Witten  theory?

\subsection{Gromov-Witten theory of families of local elliptic curves}
Let $\overline{\mathcal{M}}_{1,r}$ be the moduli space of Deligne-Mumford stable
curves of genus $1$ with $r$ markings.{\footnote{We will always assume 
$r>0$ for stability.}}
Let
$$ \pi:\mathcal{E} \rightarrow \overline{\mathcal{M}}_{1,r}\, $$
be the universal elliptic curve with sections
$$\mathsf{p}_1, \ldots, \mathsf{p}_r : \overline{\mathcal{M}}_{1,r}\rightarrow
\mathcal{E}$$
associated to the markings.
Let 
\begin{equation*}
\pi_{\mathbb{C}^2}: \mathcal{E} \times \mathbb{C}^2  \rightarrow \overline{\mathcal{M}}_{1,r}
\end{equation*}
be the universal {\em local curve} over $\overline{\mathcal{M}}_{1,r}$. The torus $\T=(\CC^*)^2$
acts on the $\com^2$ factor as before.

Let $\mu^1,\ldots,\mu^r \in \mathrm{Part}(n)$, and let
 $\Mbar^\bullet_g(\pi_{\CC^2},\mu^1,\ldots,\mu^r)$ be the moduli space
of stable{\footnote{The superscript $\bullet$ indicates possibly {\em disconnected} domain curves (but no connected component of the domain is contracted to a point). We follow the conventions of
\cite{bp}.}}
relative maps to the fibers of $\pi_{\CC^2}$,
$$\epsilon: \Mbar^\bullet_g(\pi_{\CC^2},\mu^1,\ldots, \mu^r) \rightarrow \Mbar_{1,r}\, .$$
The fiber of $\epsilon$ over 
the moduli point
$$(E,p_1,\ldots,p_r) \in \Mbar_{1,r}\, $$
is the moduli space of stable maps of genus $g$
to $E \times \com^2 $ relative to the divisors
determined by the nodes and 
the markings of $E$
with boundary{\footnote{The boundary conditions are unordered, and the cohomology weights of the boundary conditions
are all the identity class.}}
condition
$\mu^i$ over 
the divisor $p_i \times \com^2 $.  Since the degree $n$ is recorded in
the size of the partitions $\mu^i$, we omit $n$ from the notation for
$\Mbar^\bullet_g(\pi_{\CC^2},\mu^1,\ldots, \mu^r)$.

The moduli space $\Mbar^\bullet_g(\pi_{\CC^2},\mu^1,\ldots,\mu^r)$ has $\pi_{\CC^2}$-relative
virtual dimension
$$-nr+\sum_{i=1}^r \ell(\mu^i)\, .$$
The Gromov-Witten series of the family $\pi_{\CC^2}$ is defined by 
$$
\blang \mu^1, \mu^2, \ldots, \mu^r \brang^{\pi_{\CC^2},\bullet}
=\sum_{b=0}^\infty
{u^{b}}\int_{\overline{\mathcal{M}}_{1,r}} \epsilon_*\left(
\Big[\Mbar^\bullet_{g[b]}\big(\pi_{\CC^2}, \mu^1,\ldots ,\mu^r
\big)\Big]^{vir_{\pi_{\mathbb{C}^2}}}\right). $$
Here, the summation index $b$ is the branch point number, so
$$2g[b]-2= b+nr- \sum_{i=1}^r \ell(\mu^i)\, .$$

The moduli space of stable maps 
$\Mbar^\bullet_{g[b]}\big(\pi_{\CC^2}, \mu^1,\ldots ,\mu^r
\big)$
is
empty unless $g[b]$ is an integer.
The virtual class 
$\Big[\Mbar^\bullet_{g[b]}\big(\pi_{\CC^2}, \mu^1,\ldots ,\mu^r
\big)\Big]^{vir_{\pi_{\mathbb{C}^2}}}$
is the $\pi_{\CC^2}$-relative 
$\T$-equivariant
virtual class of the family of relative stable maps to the
fibers of $\pi_{\CC^2}$.
We define
\begin{eqnarray*}
\blang \mu^1, \mu^2, \ldots, \mu^r \brang_{g}^{\pi_{\CC^2},\bullet}
&= &  \text{Coeff}_{u^{2g-2}} \Big[ \blang \mu^1, \mu^2, \ldots, \mu^r \brang^{\pi_{\CC^2},\bullet}
\Big] \\
&= &
\int_{\overline{\mathcal{M}}_{1,r}} \epsilon_*\left(
\Big[\Mbar^\bullet_{g}\big(\pi_{\CC^2}, \mu^1,\ldots ,\mu^r
\big)\Big]^{vir_{\pi_{\mathbb{C}^2}}}\right). 
\end{eqnarray*}
To emphasize the degree $n$, we will sometimes use the notation
$$
\blang \mu^1, \mu^2, \ldots, \mu^r \brang_{g,n}^{\pi_{\CC^2},\bullet} =
\blang \mu^1, \mu^2, \ldots, \mu^r \brang_{g}^{\pi_{\CC^2},\bullet} \, .$$


The Gromov-Witten series of $\Hilb$ and the Gromov-Witten series of the
family $\pi_{\CC^2}$ are related by the following result of \cite{pt}.

\vspace{8pt}
\noindent {\bf{Theorem A}} (Pandharipande-Tseng).
    {\em For all 
$\mu^1,\mu^2,\ldots,\mu^r\in \text{\em Part}(n)$, we have  
$$\blang \mu^1, \mu^2, \ldots, \mu^r \brang_{1}^{\Hilb} =
(-i)^{\sum_{i=1}^r \ell(\mu^i)-|\mu^i|}\blang \mu^1, \mu^2, \ldots, \mu^r \brang^{\pi_{\CC^2},\bullet}$$
 after the variable change $-q=e^{iu}$.}
\vspace{8pt}

The calculation of the Gromov-Witten invariants of $\Hilb$  in genus 1 is therefore
{\em equivalent} to the calculation of the Gromov-Witten invariants of the
family $\pi_{\mathbb{C}^2}$ of local elliptic curves. 
For example,
\begin{equation}\label{eqn:hilb_curve_g1}
\blang D \brang_{1}^{\hilbnc} = -(-i)^{-1} \blang (2,1^{n-2})\brang^{\pi_{\CC^2},\bullet}\, .
\end{equation}


\subsection{Moduli of abelian varieties}

Let $\mathcal{A}_g$ be the
moduli space  of principally polarized abelian varieties $(X,\theta)$ of dimension $g$. 
The space $\mathcal A_g$ is a nonsingular Deligne-Mumford stack of dimension $\binom{g+1}{2}$.
Let 
$$
\nu:\X_g\rightarrow \mathcal{A}_g
$$
be the universal family of principally polarized abelian varieties.
We refer the reader to \cite{BL} for the 
foundations of the study of the moduli of
abelian varieties.

A general abelian variety $(X,\theta)$ parameterized by $\mathcal{A}_{g}$ has Picard number $1$.
The \emph{Noether-Lefschetz locus} of $\mathcal{A}_g$ parameterizes abelian varieties
with Picard number at least 2. The simplest irreducible components of the
Noether-Lefschetz locus of $\mathcal{A}_g$ are related to the geometry of
elliptic curves on abelian varieties. For $g\geq 1$ and $n\geq1 $, let 
$$
\NL_{g,n} = \left \lbrace (X, \theta) \in \mathcal{A}_g \left| \begin{array}{c}
    X\text{ contains a subgroup }E\subset X \\
    \text{ which is an elliptic curve of degree } \theta\cdot[E]=n  
\end{array}\right.\right\rbrace.
$$
Let $[{\NL}_{g,n}]\in \mathsf{CH}^{g-1}(\mathcal{A}_g)$ be the associated 
cycle class.{\footnote{ The $g=1$ case is degenerate: $[\NL_{1,n}]=[\mathcal{A}_1]$ by definition.
All Chow groups are taken with $\mathbb{Q}$-coefficients.}} The following linear combination of components
plays a geometrically important
role{\footnote{As usual, $\sigma_r(n)= \sum_{d|n} d^r$.}}
$$
[\widetilde{\NL}_{g,n}] = \sum_{n'|n} \sigma_1\left(\frac{n}{n'}\right)[\NL_{g,n'}]
\ \in \ \mathsf{CH}^{g-1}(\mathcal{A}_g)
\, ,
$$
see \cite{GreerLian}.

 The Hodge bundle is the rank $g$ vector bundle
\[
\mathbb{E}_g=\nu_{*} (\Omega_{\nu}) . 
\]
The Chern classes 
$\lambda_i= c_i(\mathbb{E}_g)\in \mathsf{CH}^i(\mathcal{A}_g)$
generate 
the {\em tautological ring} \cite{vdG}.
$$\mathsf{R}^*(\mathcal{A}_g) \subset \mathsf{CH}^*(\mathcal{A}_g)\,.$$
A canonical $\mathbb{Q}$-linear projection operator
$$\mathsf{taut}: \mathsf{CH}^*(\mathcal{A}_g) \rightarrow  \mathsf{R}^*(\mathcal{A}_g)
$$
has been constructed in \cite{cmop}.

The following result of \cite{il} determines the projections
of the simplest components of the Noether-Lefschetz loci.

\vspace{8pt}
\noindent {\bf{Theorem B}} (Iribar-L\'opez). 
{\em For $g\geq 2$ and $n\geq 1$,
we have 
$$\mathsf{taut}([\NL_{g,n}]) = \frac{n^{2g-1}g}{6|B_{2g}|}
\prod_{p|n} (1-p^{2-2g}) \lambda_{g-1}$$
or, equivalently{\footnote{Equation \eqref{aitorproj} is correct also for $g=1$ with the convention $[\NL_{1,n}]=[\mathcal{A}_1]$.}}, 
\begin{equation}\label{aitorproj}
\mathsf{taut}\left( \frac{(-1)^g}{24}\lambda_{g-1}+  \sum_{n = 1}^\infty[\widetilde{\NL}_{g,n}]\, Q^n\right) = \frac{(-1)^g}{24}E_{2g}(Q)\lambda_{g-1},
\end{equation}
where $E_{2g}(Q)$ is the Eisenstein modular function of weight $2g$ in
the variable $Q= e^{2\pi i \tau}$.}
\vspace{8pt}

A basic open question regarding the structure of the projection operator is
whether $\mathsf{taut}$ is a homomorphism of $\mathbb{Q}$-algebras:
$$\mathsf{taut}(\gamma) \cdot \mathsf{taut}(\widehat{\gamma}) \stackrel{?}{=} \mathsf{taut}(\gamma \cdot \widehat{\gamma}) \ \in \ \mathsf{R}^*(\mathcal{A}_g)
$$
for all $\gamma, \widehat{\gamma} \in \mathsf{CH}^*(\mathcal{A}_g)$.

\subsection{A triple equivalence}

\subsubsection{Hilbert schemes of points}
\label{azazaz}
Let $\mathsf{Tr}_n$ be the normalized trace of
the operator $\mathsf{M}_D^{\mathsf{Hilb}^n(\mathbb{C}^2)}$ of quantum multiplication by the divisor $D$ on 
$QH^*_\T(\Hilb)$,
$$\mathsf{Tr}_n = \frac{1}{t_1+t_2} \mathsf{trace} (\mathsf{M}_D^{\mathsf{Hilb}^n(\mathbb{C}^2)})
\, .$$
Using the formula of \cite{op} for 
$\mathsf{M}_D^{\mathsf{Hilb}^n(\mathbb{C}^2)}$,
we obtain
\begin{equation}\label{eqn:TR_n}
\mathsf{Tr}_n=\sum_{\mu\in \mathsf{Part}(n)}\sum_i\left(\frac{\mu_i^2}{2}\frac{(-q)^{\mu_i}+1}{(-q)^{\mu_i}-1}-\frac{\mu_i}{2}\frac{(-q)+1}{(-q)-1}\right)\in \mathbb{Q}(q),
\end{equation}
where the first summation 
is over the set $\mathsf{Part}(n)$
of partitions of $n$ and
the second summation is over all parts $\mu_i$ of the partition $\mu \in \mathsf{Part}(n)$.

The first of three equivalent statements is the following
calculation of the basic genus 1 Gromov-Witten invariant
of the Hilbert scheme of points.

\begin{thm} For $n\geq 1$, we have
\label{Equiv1}
\begin{equation*}
\big\langle \, D\,  \big\rangle_{1}^{\mathsf{Hilb}^n(\mathbb{C}^2)} = 
-\frac{1}{24} \frac{(t_1+t_2)^2}{t_1t_2}
\left( \mathsf{Tr}_n + \sum_{k=2}^{n-1} {\sigma_{-1}(n-k)} \mathsf{Tr}_k \right) \, .
\end{equation*}
\end{thm}

The evaluations for $n=1,2$ agree
with the known results:
$$
\big\langle \, D\,  \big\rangle_{1}^{\mathsf{Hilb}^1(\mathbb{C}^2)} = 
-\frac{1}{24} \frac{(t_1+t_2)^2}{t_1t_2}
\cdot \mathsf{Tr}_1 = 0\, ,$$ 
$$\big\langle \, D\,  \big\rangle_{1}^{\mathsf{Hilb}^2(\mathbb{C}^2)} =
-\frac{1}{24}\frac{(t_1+t_2)^2}{t_1t_2} \cdot \mathsf{Tr}_2 =
-\frac{1}{24}\frac{(t_1+t_2)^2}{t_1t_2} \cdot \frac{q+1}{q-1}\,.$$
For $n=3,4,5$, we obtain
\begin{eqnarray*}
\big\langle \, D\,  \big\rangle_{1}^{\mathsf{Hilb}^3(\mathbb{C}^2)}
&=&-\frac{1}{24}\cdot \frac{(t_1+t_2)^2}{t_1t_2}\cdot(\mathsf{Tr}_3+\mathsf{Tr}_2)\\  & = &   -\frac{1}{24}\cdot \frac{(t_1+t_2)^2}{t_1t_2} \cdot \left(\frac{5q^3-3q^2-3q+5}{(q-1)(q^2-q+1)} \right) \\
& = & 
-\frac{1}{24}\cdot\frac{(t_1+t_2)^2}{t_1t_2} \cdot \left(-5-7 q-q^{2}+2 q^{3}-q^{4}-7 q^{5}-10 q^{6}-7 q^{7}+...\right),
\end{eqnarray*}

\begin{eqnarray*}
\big\langle \, D\,  \big\rangle_{1}^{\mathsf{Hilb}^4(\mathbb{C}^2)}
&=&-\frac{1}{24}\cdot \frac{(t_1+t_2)^2}{t_1t_2}\cdot
\left(\mathsf{Tr}_4+\mathsf{Tr}_3+\frac{3}{2}\mathsf{Tr}_2\right) \\  & = &   -\frac{1}{24}\cdot \frac{(t_1+t_2)^2}{t_1t_2} \cdot 
{{\left(\frac{35q^5-28q^4+23q^3+23q^2-28q+35}{2(q-1)(q^2+1)(q^2-q+1)}\right)}}
\\
& = & -\frac{1}{24}\cdot\frac{(t_1+t_2)^2}{t_1t_2} \cdot\left(-\frac{35}{2}-21 q-q^{2}-3 q^{3}-17 q^{4}-21 q^{5}-19 q^{6}-21 q^{7}+...\right)\, ,
\end{eqnarray*}

{\tiny
\begin{eqnarray*}
\big\langle \, D\,  \big\rangle_{1}^{\mathsf{Hilb}^5(\mathbb{C}^2)}
&=&-\frac{1}{24}\cdot \frac{(t_1+t_2)^2}{t_1t_2}\cdot
\left(\mathsf{Tr}_5+\mathsf{Tr}_4+\frac{3}{2}\mathsf{Tr}_3+\frac{4}{3}\mathsf{Tr}_2\right) \\  & = &   -\frac{1}{24}\cdot \frac{(t_1+t_2)^2}{t_1t_2} \cdot 
{{\left(-\frac{272q^9-539q^8+760q^7-629q^6+302q^5+302q^4-629q^3+760q^2-539q+272}{6(q-1)(q^2+1)(q^2-q+1)(q^4-q^3+q^2-q+1)}\right)}}
\\
& = &-\frac{1}{24}\cdot \frac{(t_1+t_2)^2}{t_1t_2}\cdot
\left(\frac{136}{3}+\frac{277}{6} q-\frac{41}{6} q^{2}+\frac{17}{3} q^{3}+\frac{151}{6} q^{4}+\frac{127}{6} q^{5}+\frac{101}{3} q^{6}+\frac{277}{6} q^{7}+... \right).
\end{eqnarray*}
}

\subsubsection{Families of local elliptic curves}
It is natural to look for an analogue of Theorem \ref{Equiv1}
for the Gromov-Witten theory of families
of local elliptic curves using the  correspondence of Theorem A. 
The best statement
is expressed in terms of the {\em descendent} Gromov-Witten theory of
the family
$$\pi: \mathcal{E} \rightarrow 
\overline{\mathcal{M}}_{1,1}\, $$
for stable maps with {\em connected} domains.

Let
 $\Mbar^\circ_{g,1}(\pi,n)$ be the moduli space
of stable{\footnote{The superscript $\circ$ indicates {\em connected} domain curves.}}
relative maps to the fibers of $\pi$,
$$\epsilon: \Mbar^\circ_{g,1}(\pi,n) \rightarrow \Mbar_{1,1}\, .$$
The fiber of $\epsilon$ over 
the moduli point
$$(E,p_1) \in \Mbar_{1,1}\, $$
is the moduli space of stable maps of
1-pointed connected genus $g$ curves
to $E$ of degree $n$ 
relative to the divisors
determined by the nodes of $C$.
The moduli space $\Mbar^\circ_{g,1}(\pi,n)$ has 
$\pi$-relative
virtual dimension $2g-1$ and total virtual dimension $2g$.

For $g\geq 2$, let
$$
\blang \tau_1(\mathsf{p}_1) \lambda_g \lambda_{g-2}  \brang_{g,n}^{\pi,\circ}
= \int_{[\Mbar^\circ_{g,1}(\pi,n)]^{vir_\pi}}
\tau_1(\mathsf{p_1}) \lambda_g \lambda_{g-2} \, \in \, \mathbb{Q}\, , $$
where $\mathsf{p}_1$ is viewed as a divisor class on $\mathcal{E}$ and
the Hodge classes $\lambda_g\lambda_{g-2}$ are pulled-back from the
moduli of the domain curves.{\footnote{The $\lambda_g$
class plays an important role in the theory Hodge integrals, see
\cite{FabPan,GetP,MPS}.}}

\begin{thm}\label{Equiv2} \, For $g\geq 2$, we have
\begin{equation*}
\sum_{n=0}^\infty 
\blang \tau_1(\mathsf{p}_1) \lambda_g \lambda_{g-2}  \brang_{g,n}^{\pi,\circ} Q^n
 = 
 \frac{(-1)^g}{24}\frac{|B_{2g}|}{4g}\frac{|B_{2g-2}|}{(2g-2)! } E_{2g}(Q)\, .
\end{equation*}
\end{thm}

The $g=1$ case is degenerate. The corresponding evaluation is
\begin{equation}\label{ggg111}
\sum_{n=0}^\infty 
\blang \tau_1(\mathsf{p}_1) \brang_{1,n}^{\pi,\circ} Q^n
 = 
 \frac{(-1)}{24}\frac{|B_{2}|}{4}\frac{|B_{0}|}{0! } E_{2}(Q) =
 -\frac{1}{576} E_2(Q)
  \, .
\end{equation}
Theorem A together with a families relative/descendent correspondence
will be used to show that Theorem \ref{Equiv1} and \ref{Equiv2} are equivalent.

\subsubsection{The homomorphism question}
The third equivalence involves the homomorphism question  for 
the
projection operator
$$\mathsf{taut}: \mathsf{CH}^*(\mathcal{A}_g) \rightarrow  \mathsf{R}^*(\mathcal{A}_g)\, .
$$
Let $\mathcal{M}_g^{\mathsf{ct}}$ be the $3g-3$ dimensional moduli space of genus $g$ curves of compact type, and
let
 $$\mathsf{Tor}: \mathcal{M}_g^{\mathsf{ct}} \rightarrow \mathcal{A}_g$$
 be the Torelli map. Since $\mathsf{Tor}$ is proper, we can push forward
 the fundamental class{\footnote{The $g=1$ case is
 degenerate: $\mathsf{Tor}_*[\mathcal{M}_1^{\mathsf{ct}}]=[\mathcal{A}_1]$ by definition.}}:
 $$\mathsf{Tor}_*[\mathcal{M}_g^{\mathsf{ct}}] \in \mathsf{CH}^{{\binom{g+1}{2}}-3g+3}(\mathcal{A}_g)\, .$$
The intersection theory of the Torelli cycle has been studied in \cite{COP,D,il}.

\begin{thm} For $g\geq 1$ and $n\geq 1$, we have
\label{Equiv3}
\begin{equation*}
\mathsf{taut}(\mathsf{Tor}_*[\mathcal{M}_g^{\mathsf{ct}}] ) \cdot \mathsf{taut}([\NL_{g,n}]) \, =\,  
\mathsf{taut}(\mathsf{Tor}_*[\mathcal{M}_g^{\mathsf{ct}}] \cdot 
[\NL_{g,n}]) \ \in \ \mathsf{R}^*(\mathcal{A}_g)\, .
\end{equation*}
\end{thm}

In other words, the classes $\mathsf{Tor}_*[\mathcal{M}_g^{\mathsf{ct}}],  [\NL_{g,n}]
\in \mathsf{CH}^*(\mathcal{A}_g)$ satisfy the multiplicative property with respect
to $\mathsf{taut}$.
Theorem B together with a study of the geometry of the Torelli map was used
to show that Theorems \ref{Equiv2} and \ref{Equiv3} are equivalent in \cite{il}.

\subsubsection{Proof strategy}
A central result of the paper is the proof of Theorem \ref{Equiv2}. As a consequence of
the equivalences, 
Theorems \ref{Equiv1} and \ref{Equiv3}
will then also be proven.

While Gromov-Witten theory is well-developed for the study of a fixed target variety, 
the main difficulty in proving Theorem \ref{Equiv2} is that there are very few techniques
which are useful for the calculations of Gromov-Witten invariants in families.
Our proof of Theorem \ref{Equiv2} uses a mix of Hodge integrals, formulas for the
Gromov-Witten theory of a fixed elliptic target, and new constraints
on family invariants. The methods yields calculations of
more general families Hodge integrals in  Section 
\ref{sec:general_hodge}.

\subsection{Reduction of multi-point series to 1-point series}

By \cite[Section 4.2]{op},  the quantum powers of $D$ generate the quantum cohomology $QH^*_\T(\hilbnc)$. 
More precisely,
the set  
\begin{equation}\label{eqn:D_power_basis}
\{1, D,  D^{*2}, D^{*3},\ldots, D^{*(|\mathsf{Part}(n)|-1)}\}    
\end{equation}
is a basis of $QH^*_\T(\hilbnc)$ as a $\mathbb{Q}(t_1,t_2,q)$-vector space. Our first reduction result is that
the
genus 1 Gromov-Witten theory \eqref{eqn:g1_multi} of
$\hilbnc$ can be effectively reduced to 1-point series in the basis \eqref{eqn:D_power_basis}.

\begin{thm}\label{Red1}
To every genus $1$ series $\blang D^{* k_1},...,D^{*k_\ell}\brang^{\hilbnc}_{1}$, there are canonically
associated functions
\begin{equation}\label{ccck}
\big\{ C_{k,m} \big\}_{0\leq k \leq |\mathsf{Part}(n)|-1\, ,\,  0\leq m\leq \ell-1} \subset  \mathbb{Q}(t_1,t_2,q)    
\end{equation}
for which the following equation holds:
\begin{equation*}
\blang D^{* k_1},...,D^{*k_\ell}\rangle_{1}^{\hilbnc}=\sum_{k=0}^{|\mathsf{Part}(n)|-1}\, \sum_{m=0}^{\ell-1} \, C_{k,m}\cdot \left(q\frac{d}{dq}\right)^m\blang D^{*k}\brang_{1}^{\hilbnc}\, .    
\end{equation*}
\end{thm}

The functions \eqref{ccck} are
 constructed as
rational functions of the matrix coefficients of $\mathsf{M}^{\hilbnc}_D$.
The proof of Theorem \ref{Red1} uses the WDVV relations in genus 0 and Getzler's relation in genus 1.


\subsection{Reduction to 
$\blang D \brang_{1}^{\hilbnc}$} 
The reduction of the genus 1 Gromov-Witten theory of $\hilbnc$ to 
the basic series $\langle D \rangle^{\hilbnc}_{1}$ is more subtle.
Since the quantum cohomology of $\hilbnc$ is semisimple, the Givental-Teleman
reconstruction result \cite{g,t} reduces the full higher genus Gromov-Witten theory of
$\hilbnc$ to the genus 0 theory together with the calculation of the $\mathsf{R}$-matrix.
The $\mathsf{R}$-matrix was proven to be determined{\footnote{Since the
Gromov-Witten theory of $\hilbnc$ is equivariant, the associated $\mathsf{CohFT}$
is {\em not} conformal.}} by constant map invariants in 
\cite{pt}. The resulting control of the full Gromov-Witten theory of $\hilbnc$
was used in \cite{pt} to prove both the crepant resolution and the
GW/DT correspondences associated to the
geometry, but was not sufficient to obtain closed form evaluations of {\em any} higher
genus series (because of the lack of closed forms for the
$\mathsf{R}$-matrix).

In genus 1, the Givental-Teleman reconstruction formula 
reduces to an earlier equation due to
Givental \cite{g1} on the associated Frobenius manifold:
\begin{equation} \label{dfd}
d\mathcal{F}_1^{\hilbnc} = \frac{1}{2} 
\sum_{i=1}^{|\mathsf{Part}(n)|} \mathsf{R}_{ii}\, du_i + \frac{1}{48} \sum_{i=1}^{|\mathsf{Part}(n)|} d\log \Delta_i\, .
\end{equation}
Here, $\mathcal{F}_1^{\hilbnc}$ is the $g=1$ Gromov-Witten potential of $\hilbnc$,
the $u_i$ are the canonical coordinates, and 
$$\Delta_i = \left\langle \frac{\partial}{\partial u_i}, \frac{\partial}{\partial u_i}  \right \rangle^{-1}\, .$$
From Theorem \ref{Equiv1}, we know explicitly the functions
\begin{equation} \label{ses}
\blang \underbrace{D, \ldots, D}_{\ell} \brang_{1}^{\hilbnc} = \left( q \frac{d}{dq}\right)^{\ell-1}
\blang D \brang_{1}^{\hilbnc}
\end{equation}
for $\ell=1,\ldots, |\mathsf{Part}(n)|$.
We construct a $|\mathsf{Part}(n)|\times |\mathsf{Part}(n)|$
matrix which expresses the functions
\eqref{ses}
in terms of the functions $\mathsf{R}_{ii}\, \big|_{\mathsf{t}=0}$, where $\mathsf{t}$ is the
coordinate on the Frobenius manifold. The
non-degeneracy of the Wronskian 
\begin{equation*}
\mathsf{W}=\left(\begin{array}{cccc}
\nabla_D u_1\, \big |_{\mathsf{t}=0} & \nabla_D u_2\, \big|_{\mathsf{t}=0} &\ldots & \nabla_D u_{|\mathsf{Part}(n)|}\, |_{\mathsf{t}=0}\\[8pt]
q\frac{d}{dq} \nabla_D u_1\, \big|_{\mathsf{t}=0} & q\frac{d}{dq} \nabla_D u_2\, \big|_{\mathsf{t}=0} &\ldots & q\frac{d}{dq} \nabla_D u_{|\mathsf{Part}(n)|}\, \big|_{\mathsf{t}=0}\\[8pt]
\vdots& \vdots& \vdots& \vdots\\[8pt]
(q\frac{d}{dq})^{|\mathsf{Part}(n)|-1}\nabla_D u_1\, \big|_{\mathsf{t}=0} & (q\frac{d}{dq})^{|\mathsf{Part}(n)|-1}\nabla_D u_2\, \big|_{\mathsf{t}=0} &\ldots & (q\frac{d}{dq})^{|\mathsf{Part}(n)|-1} \nabla_D u_{|\mathsf{Part}(n)|} \, \big|_{\mathsf{t}=0}
\end{array}\right)
\end{equation*}
implies the invertibility of the system.

\begin{thm}\label{Red2}
If $\det(\mathsf{W})\neq 0$, the diagonal entries $\mathsf{R}_{ii}\, \big|_{\mathsf{t}=0}$ are determined by
$\blang D \brang_{1}^{\hilbnc}$.
\end{thm}

By Givental's equation \eqref{dfd}, the entire genus $1$ Gromov-Witten theory of $\hilbnc$
is then determined by $\blang D \brang_{1}^{\hilbnc}$ together with genus 0 data.
We have checked the non-degeneracy of the Wronskian for $\hilbnc$ for $n\leq 7$ and conjecture
the nondegeneracy for all $n$.


\begin{thm}\label{Red3}
If $\det(\mathsf{W})\neq 0$, the full genus 1 Gromov-Witten theory of $\hilbnc$ can be effectively reconstructed from $\blang D \brang_{1}^{\hilbnc}$ and $\mathsf{M}_D^{\hilbnc}$. 
\end{thm}

\subsection{Acknowledgements}

Conversations with P. Bousseau, S. Canning, L. Drakengren, J. Feusi, F. Greer, C. Lian, S. Molcho, G. Oberdieck, A. Okounkov, D. Oprea, A. Pixton, D. Ranganathan, and J. Schmitt have played an important role in our work. We thank J. Hu and Z. Qin for discussions about their paper \cite{HQ}. 

A.I.L. was supported by SNF-200020-219369. R.P. was supported by SNF-200020-219369 and Swiss\-MAP. The first presentation of the results proven here was given by R.P. at the {\em Belgian-Dutch Algebraic Geometry Seminar} in Leiden in October 2023 with very helpful discussions with C. Faber, D. Holmes, and B. Moonen afterwards. 
 H.-H. T. was supported in part by a Simons foundation collaboration grant. The project was advanced during several visits of H.-H.T. to the {\em Forschungsinstitut f\"ur Mathematik}
at ETH Z\"urich.

\section{Degree 0 invariants of \texorpdfstring{$\hilbnc$}{hilbn}}\label{sec:modularity}

We expand here explicitly the formula of Theorem \ref{Equiv1} for the degree $0$
invariants of $\hilbnc$ to show the connections to Hilbert scheme calculations of
Carlsson and Okounkov \cite{co}.

Consider first the generating series over the Hilbert schemes of points:
\begin{equation*}
\sum_{n=0}^\infty  \la D\ra_1^{\mathsf{Hilb}^{n}(\mathbb{C}^2)}\, Q^n =\la D\ra_1^{\mathsf{Hilb}^2(\mathbb{C}^2)}\, Q^2+   \la D\ra_1^{\mathsf{Hilb}^3(\mathbb{C}^2)}\, Q^3+ \la D\ra_1^{\mathsf{Hilb}^4(\mathbb{C}^2)}\, Q^4+...\, ,
\end{equation*}
where the $n=0,1$ terms vanish since $D=0$ for $n=0,1$.
Theorem \ref{Equiv1} can be written as
\begin{equation} \label{exxx}
\sum_{n=0}^\infty  \la D\ra_1^{\mathsf{Hilb}^{n}(\mathbb{C}^2)} \, Q^n =
-\frac{1}{24}\frac{(t_1+t_2)^2}{t_1t_2} \cdot \left( 1+ \sum_{n=1}^\infty \sigma_{-1}(n)\,  Q^n\right)
\cdot \left(\sum_{n\geq 2 }^\infty \mathsf{Tr}_{n}  \, Q^{n}\right)\, .
\end{equation}
The only $q$-dependence on the right side of \eqref{exxx} is via $\mathsf{Tr}_{n}$.
After restricting \eqref{eqn:TR_n}
to $q=0$, we obtain
\begin{equation} \label{constD}
\sum_{n\geq 2}^\infty \mathsf{Tr}_{n}|_{q=0}\,  Q^n =-\sum_{n=1}^\infty Q^{n}\sum_{\mu\in \mathsf{Part}(n)} \sum_i \left(\frac{\mu_i^2-\mu_i}{2}\right).
\end{equation}


%

Let  $\mathcal{P}(Q)= \prod_{k\geq 1} \frac{1}{1-Q^k}$, 
$\mathcal{E}_2(Q)=\sum_{k\geq 1}\frac{kQ^k}{1-Q^k}$,
and $\mathcal{E}_3(Q)=\sum_{k\geq 1}\frac{k^2Q^k}{1-Q^k}$.
The following identities follow easily from the generating function of $\mathrm{T}(n,a)$, the number of times the part $a$ 
occurs in all partitions of $n$, see \cite{OEIS}:
\begin{eqnarray*}
\sum_{k=1}^\infty Q^k\Bigg(\sum_{\mu\in \mathsf{Part}(k)}\sum_i \mu_i\Bigg) &= & \mathcal{P}(Q) \cdot \mathcal{E}_2(Q)\, , \\
\sum_{k=1}^\infty Q^k\Bigg(\sum_{\mu\in \mathsf{Part}(k)}\sum_i \mu_i^2\Bigg)
&=& \mathcal{P}(Q) \cdot \mathcal{E}_3(Q)\, .
\end{eqnarray*}
After combining with \eqref{constD}, we obtain
\begin{equation*}
\sum_{n=0}^\infty \mathsf{Tr}_{n}|_{q=0}\,  Q^{n} =\frac{1}{2} \mathcal{P}(Q)(\mathcal{E}_2(Q)-\mathcal{E}_3(Q))\, .    
\end{equation*}


The geometric formula for the $q=0$ term of $\la D\ra_1^{\mathsf{Hilb}^{n}(\mathbb{C}^2)}$ is
obtained via identification of the virtual fundamental class:
\begin{eqnarray*}
\int_{[\overline{\mathcal{M}}_{1,1}
(\hilbnc,0)]^{vir}}
\text{ev}_1^*(D) & = &
\int_{\overline{\mathcal{M}}_{1,1} \times \mathsf{Hilb}^{n}(\mathbb{C}^2)} D\cdot c_{\mathsf{top}}(
\mathbb{E}^*_1 \otimes T_{\mathsf{Hilb}^{n}(\mathbb{C}^2)}) \\
&=& 
-\frac{1}{24} \int_{\mathsf{Hilb}^{n}(\mathbb{C}^2)} D\cdot c_{\mathsf{top}-1}(T_{\mathsf{Hilb}^{n}(\mathbb{C}^2)})\, . 
\end{eqnarray*}
Therefore, the $q=0$ part of Theorem \ref{Equiv1} is 
equivalent to the following identity of generating functions:
\begin{equation}\label{eqn:Ctop-1_gen}
\sum_{n=0}^\infty Q^{n}\int_{\mathsf{Hilb}^{n}(\mathbb{C}^2)}D\cdot c_{\mathsf{top}-1}(T_{\mathsf{Hilb}^{n}(\mathbb{C}^2)})
= \frac{(t_1+t_2)^2}{t_1t_2}(1+\log \mathcal{P}(Q))\cdot \frac{1}{2}\mathcal{P}(Q)(\mathcal{E}_2(Q)-\mathcal{E}_3(Q))\, .
\end{equation}





Consider now the series  $\blang c_1(\mathcal{O}/\mathcal{I})\brang$ of \cite[Corollary 3]{co} where we take the line bundle $\mathcal{L}$ to be $\mathcal{O}$ with equivariant weight $m$ as in \cite[Section 2.1.1]{co}:
\begin{eqnarray*}
\blang c_1(\mathcal{O}/\mathcal{I})\brang & = & \sum_{n=0}^\infty Q^n \int_{\hilbnc} D\cdot c(T_{\mathsf{Hilb}^n(\mathbb{C}^2)}, m) \\ &=&\sum_{n=0}^\infty Q^n \int_{\hilbnc} D\cdot \left( c_{\mathsf{top}}(T_{\mathsf{Hilb}^n(\mathbb{C}^2)}) +
m\cdot c_{\mathsf{top}-1}(T_{\mathsf{Hilb}^n(\mathbb{C}^2)})+ m^2 \cdot c_{\mathsf{top}-2}(T_{\mathsf{Hilb}^n(\mathbb{C}^2)})\ldots\right)\, .
\end{eqnarray*}
The left side of (\ref{eqn:Ctop-1_gen}) is the coefficient of $m^1$ of $\la c_1(\mathcal{O}/\mathcal{I})\ra$.
By \cite[Corollary 3]{co} together with evaluations from \cite[Section 2.2.2]{co}, we have 
$$\frac{\blang c_1(\mathcal{O}/\mathcal{I})\brang}
{\blang 1 \brang } = 
\frac{1}{2} (\mathcal{E}_2(Q)-\mathcal{E}_3(Q)) \cdot 
\frac{(t_1+t_2)(t_1+m)(t_2+m)}{t_1t_2}\,, 
$$
where the series $\blang 1 \brang$ is
\begin{eqnarray*}
\blang 1 \brang  & = &  \sum_{n=0}^\infty Q^n \int_{\hilbnc} c(T_{\mathsf{Hilb}^n(\mathbb{C}^2)}, m) \\
& = & \prod_{n\geq 1} (1-Q^n)^{(\frac{m(-t_1-t_2-m)}{t_1t_2}-1)}\, 
\end{eqnarray*}
by \cite[Corollary 1]{co}. The matching \eqref{eqn:Ctop-1_gen} then follows from a simple algebraic expansion of
the $m^1$ coefficient of  $\la c_1(\mathcal{O}/\mathcal{I})\ra$.

The degree 0 term of Gromov-Witten series $\blang 1 \brang_1^{\hilbnc}$ can be studied similarly.
The geometric formula for the $q=0$ term of $\la 1\ra_1^{\mathsf{Hilb}^{n}(\mathbb{C}^2)}$ is
\begin{eqnarray*}
\int_{[\overline{\mathcal{M}}_{1,1}
(\hilbnc,0)]^{vir}}
\text{ev}_1^*(1) & = &
\int_{\overline{\mathcal{M}}_{1,1} \times \mathsf{Hilb}^{n}(\mathbb{C}^2)}  c_{\mathsf{top}}(
\mathbb{E}^*_1 \otimes T_{\mathsf{Hilb}^{n}(\mathbb{C}^2)}) \\
&=& 
-\frac{1}{24} \int_{\mathsf{Hilb}^{n}(\mathbb{C}^2)}  c_{\mathsf{top}-1}(T_{\mathsf{Hilb}^{n}(\mathbb{C}^2)})\, , 
\end{eqnarray*}
which equals the $m^1$ coefficient of $-\frac{1}{24}\blang 1 \brang$. A calculation then yields 
\begin{equation*}
\text{Coeff}_{m^1}\Big[\la 1\ra \Big] =
\frac{t_1+t_2}{t_1t_2}\cdot \mathcal{P}(Q) \log \mathcal{P}(Q)\, .
\end{equation*}
We obtain the evaluation
\begin{equation*}
\sum_{n=0}^\infty  \la 1\ra_1^{\mathsf{Hilb}^{n}(\mathbb{C}^2)}|_{q=0}\, Q^n = -\frac{1}{24}
\frac{t_1+t_2}{t_1t_2}\cdot \mathcal{P}(Q) \log \mathcal{P}(Q)\, .
\end{equation*}

The Gromov-Witten series $\la 1\ra_1^{\mathsf{Hilb}^{n}(\mathbb{C}^2)}$ is much simpler than
$\la D\ra_1^{\mathsf{Hilb}^{n}(\mathbb{C}^2)}$. Because
of the axiom of the fundamental class,
{\em all} positive degree terms of 
$\la 1\ra_1^{\mathsf{Hilb}^{n}(\mathbb{C}^2)}$
vanish. Therefore,
\begin{equation*}
 \la 1\ra_1^{\mathsf{Hilb}^{n}(\mathbb{C}^2)} = 
 \text{Coeff}_{Q^n}\Big[
 -\frac{1}{24}
\frac{t_1+t_2}{t_1t_2}\cdot \mathcal{P}(Q) \log \mathcal{P}(Q)\Big]\, .
\end{equation*}

\section{Families Hodge integrals}
\label{sec:comparison_vir_pf}

\subsection{Overview}
Our first result is the calculation of the families Hodge integral of 
Theorem  \ref{Equiv2} over the
moduli spaces of stable maps with connected domains,
$$\epsilon: \Mbar^\circ_{g,1}(\pi,n) \rightarrow \Mbar_{1,1}\, , \  \ \ 
\pi: \mathcal{E} \rightarrow 
\overline{\mathcal{M}}_{1,1}\, . $$
After evaluation of the special $n=1$ and $g=1$ cases by hand,
the main motivation is to use the vanishing of the virtual class of the moduli space of
stable maps to a $K3$ surface to prove Theorem \ref{Equiv2}.

\subsection{The \texorpdfstring{$n=0$}{} case}
The $n=0$ invariant  concerns the moduli space of degree 0 stable maps $\Mbar^\circ_{g,1}(\pi,0)$.
After imposing the evaluation condition on the marking, the moduli space is
$$\Mbar^\circ_{g,1}(\pi,0) \supset \text{ev}_1^{-1}(\mathsf{p}_1) = \Mbar_{g,1}  \times \Mbar_{1,1}\, $$
with obstruction bundle $\mathbb{E}_g^\vee \otimes \mathbb{T}$,
where $\mathbb{T}$ is the tangent line on $\Mbar_{1,1}$ associated
to the marking.

\vspace{8pt}
\noindent $\bullet$
For $g\geq 2$, we evaluate the $n=0$ terms by:
\begin{eqnarray*}
    \blang \tau_1(\mathsf{p}_1) \lambda_g \lambda_{g-2}  \brang_{g,0}^{\pi,\circ} & = &
    \int_{[\Mbar_{g,1} \times \Mbar_{1,1}]^{vir}}\psi_1 \lambda_g \lambda_{g-2} \\
    & = & \int_{\Mbar_{g,1} \times \Mbar_{1,1}}\psi_1 \lambda_g \lambda_{g-2}
    \cdot \mathsf{e}(\mathbb E_g^\vee \otimes \mathbb{T}) \\
    &  = & (-1)^g \int_{\Mbar_{g,1} \times \Mbar_{1,1}} \psi_1\lambda_g\lambda_{g-2}(\lambda_g +\lambda_{g-1}\boxtimes\psi)\\
    &= &\frac{(-1)^g(2g-2)}{24}\frac{|B_{2g}||B_{2g-2}|}{4g(2g-2)(2g-2)!} \\
&= &\frac{(-1)^g}{24}\frac{|B_{2g}|}{4g}\frac{|B_{2g-2}|}{(2g-2)!} \, .
\end{eqnarray*}
Here, $\psi_1$ denotes the cotangent line at the marking of $\Mbar_{g,1}$.
In the third equality, we have denoted 
the dual of $c_1(\mathbb{T})$ by $\psi$, the cotangent
line at the marking of $\Mbar_{1,1}$.
In the fourth equality, we have used the dilaton equation and the 
Hodge integral evaluated in 
\cite[Theorem 4]{fphodge}. The final evaluation agrees with Theorem \ref{Equiv2} since the
Eisenstein series start with 1,
\begin{equation} \label{eisen}
E_{2g}(Q) = 1 - \frac{4g}{B_{2g}} \sum_{n=1}^\infty \sigma_{2g-1}(n) Q^n\, .
\end{equation}

\noindent $\bullet$
For $g= 1$, we evaluate the $n=0$ term by:
\begin{eqnarray*}
    \blang \tau_1(\mathsf{p}_1)  \brang_{1,0}^{\pi,\circ} & = &
    \int_{[\Mbar_{1,1} \times \Mbar_{1,1}]^{vir}}\psi_1  \\
    & = & \int_{\Mbar_{1,1} \times \Mbar_{1,1}}\psi_1 
    \cdot \mathsf{e}(\mathbb E_1^\vee \otimes \mathbb{T}) \\
    &  = & (-1) \int_{\Mbar_{1,1} \times \Mbar_{1,1}} \psi_1 (\lambda_1 +\lambda_0\boxtimes\psi)\\
    &= &-\frac{1}{576}\,  ,
\end{eqnarray*}
which agrees with the constant term on the right side of \eqref{ggg111}.

\subsection{The \texorpdfstring{$g=1$}{} case} 
For $g=1$ and $n\geq 1$, we can evaluate the integral
 $\blang \tau_1(\mathsf{p}_1)  \brang_{1,n}^{\pi,\circ}$ by hand. There are no branch points,
 so the cotangent line on the domain is pulled-back from the cotangent line
 of $\Mbar_{1,1}$. Using the latter cotangent line class,
 we can express $\blang \tau_1(\mathsf{p}_1)  \brang_{1,n}^{\pi,\circ}$
 as $\frac{1}{24}$ times a Gromov-Witten invariant
of maps to a fixed elliptic target $(E,p_1)$:
$$\blang \tau_1(\mathsf{p}_1)  \brang_{1,n}^{\pi,\circ} = \frac{1}{24} 
  \blang \tau_0({p}_1) 
  \brang_{1,n}^{E,\circ} \, .$$
Using the well-known evaluation
  $$\sum_{n=0}^\infty \blang \tau_0({p}_1)  \brang_{1,n}^{E,\circ} =
  -\frac{1}{24} E_2(Q)\, ,$$
  we deduce \eqref{ggg111} for all{\footnote{The $n=0$ term $\blang \tau_1(\mathsf{p}_1)  \brang_{1,0}^{\pi,\circ}$ also matches
 $\frac{1}{24} 
  \blang \tau_0({p}_1) 
  \brang_{1,0}^{E,\circ}$.}}
  $n\geq 1$.

\subsection{Proof of Theorem \ref{Equiv2} for \texorpdfstring{$g\geq 2$}{} and \texorpdfstring{$n>0$}{}}

\label{eeeeee}
Let $S$ be an  elliptically fibered $K3$ surface $S$ with a section{\footnote{
The section is denoted by $\mathsf{p}_1$ following the conventions related to families of points curves. The class of $\mathsf{p_1}$ is a {\em divisor} class
in the total space $S$ of the family.}}
$\mathsf{p}_1$, 
\begin{equation*}
\xymatrix{
S\ar[r]^{\pi_S} & \mathbb{P}^1\ar@/{}^{1pc}/[l]^{\mathsf{p}_1}\, ,
}    
\end{equation*}
and $24$ nodal fibers $R_1, \ldots, R_{24}\subset S$.
The fibers of $\pi_S$ are $1$-pointed genus $1$ stable curves. The map
$$\mathbb{P}^1\longrightarrow \overline{\mathcal M}_{1,1}$$
induced by $\pi_S$ is of degree $48$. Therefore,
\begin{equation*}
\int_{[\overline{\mathcal M}^\circ_{g,1}(\pi,n)]^{vir}}\tau_1(\mathsf{p}_1) \lambda_{g}\lambda_{g-2}=\frac{1}{48}\int_{[\overline{\mathcal M}^\circ_{g,1}(\pi_S,n)]^{vir}}\tau_1(\mathsf{p}_1) \lambda_{g}\lambda_{g-2}.
\end{equation*}
The moduli space of stable maps to the fibers
of 
$\pi_S: S\rightarrow \mathbb{P}^1$
lies over $\mathbb {P}^1$,
$$\epsilon_S: \Mbar^\circ_{g,1}(\pi_S,n) \rightarrow \mathbb{P}^1\, .$$


\begin{prop}\label{prop:virtual_class_comparison}
The following vanishing holds for $g\geq 2$ and $n>0$:
    $$
\int_{[\overline{\mathcal M}^\circ_{g,1}(\pi_S,n)]^{vir}}\tau_1(\mathsf{p_1})\lambda_{g-2} \cdot  
    \mathsf{e}\left(\mathbb E_g^\vee \otimes 
    \epsilon_S^*(\mathsf{Tan}_{\mathbb P^1})\right) = 0\,.$$
\end{prop}

The motivation for Proposition \ref{prop:virtual_class_comparison} comes from
the vanishing of Gromov-Witten invariants
for $K3$ surfaces in non-zero curve classes.
The Euler class 
$\mathsf{e}\left(\mathbb E_g^\vee \otimes 
    \epsilon_S^*(\mathsf{Tan}_{\mathbb P^1})
    \right)$
in the
integrand
relates the families virtual class for $\pi_S$
to the virtual class for maps to $S$.
A proof using deformation to the normal cones
of the nodal fibers $R_i\subset S$ requires a subtle
study of the logarithmic degeneration formula.
We will prove the vanishing of Proposition
\ref{prop:virtual_class_comparison} via
another path in Section \ref{sec:virtual_class_comparison}.

Proposition \ref{prop:virtual_class_comparison} allow us to exchange the families
Hodge integral
$$\int_{[\overline{\mathcal M}^\circ_{g,1}(\pi_S,n)]^{vir}}\tau_1(\mathsf{p}_1) \lambda_{g}\lambda_{g-2}$$
for a fixed target Gromov-Witten invariant. We first expand the Euler class as
\begin{equation} \label{lddr}
\mathsf{e}\left(\mathbb E_g^\vee \otimes 
    \epsilon_S^*(\mathsf{Tan}_{\mathbb P^1})\right)
=(-1)^g\lambda_g+(-1)^{g-1}\lambda_{g-1}\cdot \epsilon_S^*(2[\mathsf{pt}])\, ,
\end{equation}
where $[\mathsf{pt}]\in \mathsf{CH}_0(\mathbb{P}^1)$ is the point class.
By Proposition \ref{prop:virtual_class_comparison}, we obtain
\begin{equation*}
\int_{[\overline{\mathcal M}^\circ_{g,1}(\pi_S,n)]^{vir}}\tau_1(\mathsf{p_1})\lambda_{g-2} 
\cdot \Big( (-1)^g\lambda_g+(-1)^{g-1}\lambda_{g-1}\cdot 2\epsilon_S^*([\mathsf{pt}]))\Big)
= 0\, .
\end{equation*}
Hence,
\begin{eqnarray*}
\int_{[\overline{\mathcal M}^\circ_{g,1}(\pi_S,n)]^{vir}}\tau_1(\mathsf{p_1})\lambda_g\lambda_{g-2}& = & 
2 \int_{[\overline{\mathcal M}^\circ_{g,1}(\pi_S,n)]^{vir}} \tau_1(\mathsf{p_1})\lambda_{g-1}\lambda_{g-2} \cdot [\mathsf{pt}] \\
& = & 
2\int_{[\overline{\mathcal M}^\circ_{g,1}(E,n)]^{vir}}\tau_1(p_1)\lambda_{g-1} \lambda_{g-2}\, ,
\end{eqnarray*}
where the last integral is the Gromov-Witten invariant with a fixed elliptic curve
target $(E,p_1)$.

The evaluation of the required integral for $(E,p_1)$ follows from the methods
of \cite{OP23},
\begin{equation}\label{eqn:fixed_ell_eval}
\int_{[\overline{\mathcal M}^\circ_{g,1}(E,n)]^{vir}}\tau_1(p_1)\lambda_{g-1} \lambda_{g-2} = \frac{|B_{2g-2}|\sigma_{2g-1}(n)}{(2g-2)!}\, ,
\end{equation}
as we will explain in Section \ref{sec:general_hodge}.
The integral can also be obtained using the study of the Gromov-Witten theory of target curves \cite{op_curve1} and was first calculated by Pixton in \cite{Pix}.

Theorem \ref{Equiv2} then follows from 
$$\int_{[\overline{\mathcal M}^\circ_{g,1}(\pi,n)]^{vir}}\tau_1(\mathsf{p}_1) \lambda_{g}\lambda_{g-2}
=\frac{1}{24} \frac{|B_{2g-2}|\sigma_{2g-1}(n)}{(2g-2)!}\, $$
and the definition of the Eisenstein series \eqref{eisen}. \qed

\subsection{Proof of Proposition \ref{prop:virtual_class_comparison}}\label{sec:virtual_class_comparison}
The moduli space of stable maps to the fibers
of 
$\pi_S: S\rightarrow \mathbb{P}^1$
lies over $\mathbb {P}^1$,
$$\epsilon_S: \Mbar^\circ_{g,1}(\pi_S,n) \rightarrow \mathbb{P}^1\, .$$
The universal curve $\mu: \mathcal{C}
\rightarrow \overline{\mathcal M}^\circ_{g,1}(\pi_S,n)$
carries a universal evaluation map 
\[
\begin{tikzpicture}[scale=2, every node/.style={font=\small}]
  \node (C) at (0,1.3) {$\mathcal{C}$};
  \node (T) at (1.5,1.3) {$\mathcal{T}$};
  \node (M) at (0.75,0) {$\overline{\mathcal{M}}^\circ_{g,1}(\pi_S,n)$};

  \draw[->] (C) -- (T) node[midway, above] {$f$};
  \draw[->] (C) -- (M) node[midway, left] {$\mu$};
  \draw[->] (T) -- (M) node[midway, right] {$\nu$};
\end{tikzpicture}
\]
to the universally expanded target
\[
\begin{tikzpicture}[scale=2, every node/.style={font=\small}]
  \node (C) at (0,1.3) {$\mathcal{T}$};
  \node (T) at (1.5,1.3) {$\overline{\mathcal M}^\circ_{g,1}(\pi_S,n)\times _{\mathbb{P}^1} S$};
  \node (M) at (0.75,0) {$\overline{\mathcal{M}}^\circ_{g,1}(\pi_S,n)\, .$};

  \draw[->] (C) -- (T) node[midway, above] {$h$};
  \draw[->] (C) -- (M) node[midway, left] {$\nu$};
  \draw[->] (T) -- (M) node[midway, right] {$\pi_M$};
\end{tikzpicture}
\]
The target $\mathcal{T}$ is a family of elliptic
curves over $\overline{\mathcal M}^\circ_{g,1}(\pi_S,n)$
with possible expansion over the
24 nodal fibers $R_i$. The only permitted
expansions over $R_i$ are simple circuits of rational curves.

The relative dualizing sheaf $\omega_{\pi_S}$ of the of elliptic fibration
$\pi_S: S\rightarrow \mathbb{P}^1$
is pulled-back from the base, $$\omega_{\pi_S} \cong  \pi_S^*(\mathsf{Tan}_{\mathbb{P}^1})\, .$$
Since $g$ is an isomorphism (except for
collapsing chains of unstable rational curves
to nodes), 
$$\omega_\nu \, \cong\,  
h^* \omega_{\pi_M} \, \cong\, \nu^* \epsilon_S^*(\mathsf{Tan}_{\mathbb{P}^1}) 
\, .$$
We therefore obtain $\nu_*\omega_\nu \cong
\epsilon_S^*(\mathsf{Tan}_{\mathbb{P}^1})$.

Consider next the stabilization map
\[
\begin{tikzpicture}[scale=2, every node/.style={font=\small}]
  \node (C) at (0,1.3) {$\mathcal{C}$};
  \node (T) at (1.5,1.3) {$\mathcal{C}_{\mathsf{st}}$};
  \node (M) at (0.75,0) {$\overline{\mathcal{M}}^\circ_{g,1}(\pi_S,n)\, .$};

  \draw[->] (C) -- (T) node[midway, above] {$\mathsf{st}$};
  \draw[->] (C) -- (M) node[midway, left] {$\mu$};
  \draw[->] (T) -- (M) node[midway, right] {$\mu_{\mathsf{st}}$};
\end{tikzpicture}
\]
By the geometry of relative stable maps, the contraction $\mathsf{st}$ is an isomorphism (except for
collapsing chains of unstable rational curves
to nodes). Therefore, 
$$\omega_\mu \cong \mathsf{st}^*\omega_{\mu_{\mathsf{st}}}, $$
and we obtain $\mu_*\omega_\mu \cong \mathbb{E}_g$.

Since the moduli space $\overline{\mathcal M}^\circ_{g,1}(\pi_S,n)$ parametrizes relative maps to the
fibers of $\pi_S$, we have a pull-back map
\begin{equation} \label{injinj}
\epsilon_S^*(\mathsf{Tan}_{\mathbb{P}^1}) \cong
\nu_*\omega_{\nu} \ \stackrel{f^*}{\longrightarrow}\  \mu_*\omega_{\mu} \cong \mathbb{E}_g
\end{equation}
over $\overline{\mathcal M}_{g,1}(\pi_S,n)$.
Since $n>0$, the map \eqref{injinj} is injective,
so we obtain an exact sequence 
\begin{equation} \label{vrrf}
0 \rightarrow
\epsilon_S^*(\mathsf{Tan}_{\mathbb{P}^1}) \rightarrow
\mathbb{E}_g
\rightarrow
\mathbb{F} \rightarrow 0\, ,
\end{equation}
where $\mathbb{F}$ is a rank $g-1$ vector bundle on
$\overline{\mathcal M}^\circ_{g,1}(\pi_S,n)$.
We therefore have a factorization
\begin{equation}
\lambda_g  \ = \ c_1(\epsilon_S^*\big(\mathsf{Tan}_{\mathbb{P}^1})\big)\cdot c_{g-1}(\mathbb{F})  
\ = \ 2 \epsilon_S^*(\mathsf{pt}) \cdot \lambda_{g-1}\, 
\label{faccc}
\end{equation}
on 
$\overline{\mathcal M}_{g,1}(\pi_S,n)$.
For the second equality, the restriction of
\eqref{vrrf} to a fiber of $\epsilon_S$ is
used  to obtain 
$$c_{g-1}(\mathbb{F})|_{\epsilon_S^{-1}(\mathsf{pt})}=\lambda_{g-1}\, .$$
The factorization \eqref{faccc} of $\lambda_g$
implies the
vanishing
$$
\int_{[\overline{\mathcal M}^\circ_{g,1}(\pi_S,n)]^{vir}}\tau_1(\mathsf{p_1})\lambda_{g-2} \cdot  
    \mathsf{e}(\mathbb E_g^\vee \otimes \mathsf{Tan}_{\mathbb P^1}) = 0\ $$
by \eqref{lddr}. \qed

\vspace{8pt}

\subsection{Hodge integrals and Gromov-Witten theory for a fixed elliptic target}\label{sec:general_hodge}

The proof of Proposition \ref{prop:virtual_class_comparison}
yields
a stronger statement for $n>0$:
$$
    \mathsf{e}(\mathbb E_g^\vee \otimes \mathsf{Tan}_{\mathbb P^1})\cap [\overline{\mathcal M}^\circ_{g,r}(\pi_S,n)]^{vir} =0
$$
in $\operatorname{\mathsf{CH}}^*(\Mbar_{g,r})$. If $\mathsf{F}(\lambda)$ is
    a homogeneous{\footnote{The class $\lambda_i$ has degree $i$.}} 
    polynomial in Hodge classes
    $\lambda_i$  satisfying
    $$ \sum_{i=1}^r k_i 
+    \text{deg}(\mathsf{F}) 
    = g-1\, .
    $$
we obtain (as in Section \ref{eeeeee}):
\begin{equation}\label{eqn:moving_fixed_general}
    \left\langle
\prod_{i=1}^m \tau_{k_i}(\mathsf{p_1}) \prod_{j=m+1}^r \tau_{k_j+1}(1)\cdot
 \lambda_{g}\mathsf{F}(\lambda)
\right\rangle_{g,n}^{\pi,\circ}=\frac{1}{24} 
\left\langle
\prod_{i=1}^m \tau_{k_i}(p_1)
\cdot \prod_{j=m+1}^r \tau_{k_j+1}(1)\cdot \lambda_{g-1} 
\mathsf{F}(\lambda)\right\rangle_{g,n}^{E,\circ}\, ,
\end{equation}
where the $n=0$ case follows from an application
of Mumford's formula.

The  Gromov-Witten
Hodge integral of a fixed elliptic curve target $(E,p_1)$ on the right side can be
effectively computed. By \cite{OP23}, we have an equality of cycles
$$
\sum_{n=0}^{\infty }Q^n [\Mbar^\circ_{g,m}(E,n)]^{vir}\cdot \prod_{i=1}^m\tau_0(\mathsf{p}_1) \cdot \lambda_{g-1} = \frac{(-1)^g(2g-1)!}{(2g-2+m)!}\lambda_g\lambda_{g-1}\sum_{i=1}^m\prod_{j\neq i} \psi_j\left( Q\frac{d}{dQ}\right)^{m-1} E_{2g}(Q)
$$
in $H^{*}(\Mbar_{g,m})$. Therefore, the right hand side of \eqref{eqn:moving_fixed_general} is equal to the coefficient of $Q^n$ in
$$
\frac{(-1)^g(2g-1)!}{24(2g+m-2)!}\left( Q\frac{d}{dQ}\right)^{m-1} E_{2g}(Q)\sum_{i=1}^m\int_{\Mbar_{g,r}} \prod_{j=1}^r\psi_j^{k_j + 1-\delta_{ij}}\cdot \mathsf{F}(\lambda)\lambda_g \lambda_{g-1}\, .
$$
The integrals over $\Mbar_{g,r}$ can be evaluated effectively via
Hodge integral techniques.

\vspace{8pt}
\noindent $\bullet$ For $\mathsf{F} =1$, we have the exact evaluation
$$
\int_{\Mbar_{g,r}} \prod_{j=1}^r\psi_j^{k_j + 1-\delta_{ij}}\cdot\lambda_g \lambda_{g-1} =\frac{|B_{2g}|}{2^{2g-1}(2g)!}\frac{(2g+r-3)!(2k_i+1)}{(2k_1+1)!!\ldots (2k_r+1)!!}\, ,
$$
given by the Virasoro constraints of $\mathbb P^2$ \cite{GetP}. We arrive at the following:
\begin{equation*}
    \sum_{n=0}^\infty Q^n\left\langle
\prod_{i=1}^m \tau_{k_i}(\mathsf{p_1}) \prod_{j=m+1}^r \tau_{k_j+1}(1)
\cdot \lambda_{g}
\right\rangle_{g,n}^{\pi,\circ}=C\frac{(-1)^g}{24}\frac{|B_{2g}|}{4g}\left( Q\frac{d}{dQ}\right)^{m-1} E_{2g}(Q)\, ,
\end{equation*}
where
$$
C=\frac{(2g+r-3)!\sum_{i=1}^{m}(2k_i+1)}{2^{2g-2}(2g+m-2)!(2k_1+1)!!\ldots (2k_r+1)!!}\, .
$$

\vspace{8pt}
\noindent $\bullet$
For $\mathsf{F}= \lambda_{g-2}$ and $r=m=1$,  we obtain \eqref{eqn:fixed_ell_eval}:
$$
\langle \tau_1(p_1)\lambda_{g-1}\lambda_{g-2}\rangle^{E,\circ}_{g,n} = (-1)^{g-1}\frac{4g }{B_{2g}}\sigma_{2g-1}(n)\int_{\Mbar_{g,1}}\psi_1 \lambda_g \lambda_{g-1}\lambda_{g-2} = \frac{|B_{2g-2}|}{(2g-2)!}\sigma_{2g-1}(n)\, .
$$

\vspace{8pt}
\noindent Further development of these ideas will appear in \cite{IPT}.

\section{The invariant \texorpdfstring{$\blang D \brang_{1}^{\hilbnc}$}{}}

\subsection{Overview}
We present here the proof of Theorem \ref{Equiv1}:\begin{equation*}
\big\langle \, D\,  \big\rangle_{1}^{\mathsf{Hilb}^n(\mathbb{C}^2)} = 
-\frac{1}{24} \frac{(t_1+t_2)^2}{t_1t_2}
\left( \mathsf{Tr}_n + \sum_{k=2}^{n-1} {\sigma_{-1}(n-k)} \mathsf{Tr}_k \right) \, .
\end{equation*}
Our strategy is to convert
the invariant $\big\langle \, D\,  \big\rangle_{1}^{\mathsf{Hilb}^n(\mathbb{C}^2)}$
to a families Hodge integral. After several
steps related to the connected/disconnected calculus and descendent/relative correspondence, we will show that Theorem \ref{Equiv1} follows from Theorem \ref{Equiv2}.

\subsection{Connected/disconnected calculus}\label{sec:conn_disconn_calc}
The genus 1 Gromov-Witten invariants of ${\Hilb}$ are expressed in terms of families invariants
with 
possibly {\em disconnected}{\footnote{The superscript $\bullet$ indicates possibly {\em disconnected} domain curves (but no connected
component of the domain is contracted to a point).}}
domain curves
by 
by (\ref{eqn:hilb_curve_g1}):
$$
\la \,D\,\ra_1^{\Hilb}=-\la (2,1^{n-2})\ra_1^{\Hilb}=\frac{1}{i}\la (2,1^{n-2})\ra^{\pi_{\CC^2},\bullet}\, .
$$
We will transform $\la (2,1^{n-2})\ra^{\pi_{\CC^2},\bullet}$ to
invariants with {\em connected} domain curves. The connected/dis\-connected
correspondence is well-known for the Gromov-Witten theory of a fixed target. For families invariants, new aspects appear.

There are two connected Gromov-Witten invariants ((i) and (iii)) and two partition counts ((ii) and (iv)) 
which occur in the connected/disconnect calculus for $\la (2,1^{n-2})\ra^{\pi_{\CC^2},\bullet}$:

\vspace{8pt}
\noindent ({\bf{i}}) Let $\overline{\mathcal M}^\circ_g(\pi,(2,1^{n-2}))$ be the moduli space of degree $n$ stable relative
maps to the fibers of the universal elliptic curve
$$\pi: \mathcal{E} \rightarrow 
\overline{\mathcal{M}}_{1,1}$$
with connected{\footnote{The superscript $\circ$ denotes connected domains.}} 
domains of genus $g$ and relative condition at the marking given by $(2,1^{n-2})$. Let
\begin{equation*}
\la \lambda_{g-2}\lambda_g|(2,1^{n-2})\ra_{g,n}^{\pi,\circ}=\int_{[\overline{\mathcal M}^\circ_g(\pi,(2,1^{n-2}))]^{vir_\pi}}\lambda_{g-2}\lambda_g\, .    
\end{equation*}
The descendent/relative correspondence via the degeneration to the normal cone of the
section $\mathsf{p}_1$ yields the following result proven in Section \ref{sec:degeneration_pf}.  
\begin{prop}\label{prop:degeneration_form} For $g\geq 2$ and $n\geq 2$, we have
    \begin{equation*} 
\la \tau_1(\mathsf{p}_1)\lambda_{g-2}\lambda_g\ra_{g,n}^{\pi,\circ}=\frac{\sigma_1(n)}{24}\frac{|B_{2g-2}|}{(2g-2)!} +\la \lambda_{g-2}\lambda_g|(2,1^{n-2})\ra_{g,n}^{\pi,\circ}\, . 
\end{equation*}
\end{prop}

Since $\la \tau_1(\mathsf{p}_1)\lambda_{g-2}\lambda_g\ra_{g,n}^{\pi,\circ}$
was calculated in Theorem \ref{Equiv2}, Proposition \ref{prop:degeneration_form}
completely determines the invariant $\la \lambda_{g-2}\lambda_g|(2,1^{n-2})\ra_{g,n}^{\pi,\circ}$.

The $g=1$ case takes a special form. For $n\geq 2$, we easily obtain
\begin{equation*}
\la \tau_1(\mathsf{p}_1)\ra_{1,n}^{\pi,\circ}=\frac{\sigma_1(n)}{24}\frac{|B_{0}|}{0!} +\la (2,1^{n-2})\ra_{1,n}^{\pi,\circ}\, .    
\end{equation*}
Using the evaluation \eqref{ggg111} of $\la \tau_1(\mathsf{p}_1)\ra_{1,n}^{\pi,\circ}$, we see
$$
\la (2,1^{n-2})\ra_{1,n}^{\pi,\circ} =0
$$
for $n\geq 2$.
 

\vspace{8pt}
\noindent ({\bf{ii}})
For an integer $l\geq 1$, 
let $\mathsf{Part}(l)$ be the number of partitions of $l$. 
Partitions arise naturally in the Gromov-Witten theory of $E$: 
 $\mathsf{Part}(l)$ is the count of possibly disconnected unramified covers of $E$ of degree $l$ where each cover is weighted by the order of the automorphism group. The corresponding generating series 
 is 
$$ \mathcal{P}(x) = 1 + \sum_{l=1}^{\infty}  \mathsf{Part}(l)\,  x^l = \prod_{l=1}^\infty \frac{1}{1-x^l}\, .$$

\vspace{8pt}
\noindent ({\bf{iii}}) Let $E$ be a fixed elliptic curve. 
We denote by 
$
\la (2,1^{n-2}) \ra_{g,n}^{E\times \mathbb{C}^2, \circ}    
$
the
connected genus $g$, degree $n$, $\T$-equivariant Gromov-Witten invariant of $E\times \mathbb{C}^2$ with relative condition at the divisor $\{p_1\}\times \mathbb{C}^2$ given by $(2, 1^{n-2})$. 
The connected/disconnected correspondence here is
\begin{equation}\label{eqn:conn_disc_fixed}
\la (2,1^{n-2}) \ra_{g,n}^{E\times \mathbb{C}^2, \bullet}=\sum_{2\leq m\leq n}\la (2,1^{m-2}) \ra_{g,m}^{E\times \mathbb{C}^2, \circ}\cdot \mathsf{Part}(n-m)\, .
\end{equation}
The disconnected invariants $\la (2,1^{n-2}) \ra_{g,n}^{E\times \mathbb{C}^2, \bullet}$ can be calculated by degenerating $E$ to a curve of arithmetic genus $1$ with a unique node and applying the  correspondence between local Gromov-Witten theory of $\mathbb{P}^1$ and quantum cohomology ring of  $\Hilb$ \cite{bp,op}:
\begin{equation}\label{eqn:GWHilb_fixed}
-\sum_{g\in \mathbb{Z}} u^{2g-3} \la (2,1^{n-2})\ra_{g,n}^{E\times \mathbb{C}^2,\bullet}=(-i)\cdot \mathsf{trace}\left(\mathsf{M}_D^{\mathsf{Hilb}^n(\mathbb{C}^2)}(q)\right)=(-i)\cdot \mathsf{Tr}_n\cdot (t_1+t_2),   \end{equation}
after $-q=e^{i u}$.

\vspace{8pt} \noindent ({\bf{iv}})
 For an integer $l\geq 1$, let $\widetilde{\mathsf{Part}}(l)$ be the count of possibly disconnected unramified covers of $E$ of degree $l$ where each cover is weighted by the order of automorphism group {\em and} the number of connected components. 
 For example,
$\widetilde{\mathsf{Part}}(1)=1$ and 
$$\widetilde{\mathsf{Part}}(2)=1+3/2=5/2\, ,$$
where $1=(1/2)\cdot 1\cdot 2$ is the contribution of the disconnected cover 
and $3/2=(1/2)\cdot 3\cdot 1$ is the contribution of the connected covers.
The generating series is
$$
\widetilde{\mathcal{P}}(x) = \sum_{l=1}^{\infty}  \widetilde{\mathsf{Part}}(l)\,  x^l\, . $$

\begin{lem}\label{lem:Ptilde_P}
    The series $\widetilde{\mathcal{P}}$ is determined by the equation 
    $\widetilde{\mathcal{P}} = \mathcal{P} \operatorname{log} \mathcal{P}$.
\end{lem}
\begin{proof}
    Let $\mathsf{Hur}(l,k)$ be the automorphism-weighted count
    of possibly disconnected unramified
    covers of an elliptic curve of degree $l$  with exactly $k$ connected component, and let
    $$
    \mathcal{F}(x,y) = 1 + \sum_{l=1}^\infty \sum_{k=1}^\infty  \operatorname{\mathsf{Hur}}(l,k)\,  x^l y^k\, .
    $$
    Since $\operatorname{log}(\mathcal{P}(x))$ is the generating series of connected unramified covers of an elliptic curve,
    $$
    \mathcal{F}(x,y) = \operatorname{exp}(y \operatorname{log}(\mathcal{P}(x)))\, .
    $$
    We then obtain
    $\widetilde{\mathcal{P}}(x) = \partial_y \mathcal{F} (x,y)|_{y=1} = \mathcal{P}(x) \operatorname{log}\mathcal{P}(x)$
    by the definition of $\mathsf{Hur}(l,k)$.
\end{proof}

We can now state the main connected/disconnected equation which will play an essential role
in the proof of Theorem \ref{Equiv1}. The proof will be presented in Section \ref{sec:con_discon_pf}.

\begin{prop}[Connected/disconnected calculus]\label{prop:con_discon}
For $g\geq 2$ and $n \geq 2$, we have
\begin{eqnarray*}
\la (2,1^{n-2})\ra_{g ,n}^{\pi_{\mathbb{C}^2},\bullet} &= & 
\frac{(t_1+t_2)^2}{t_1t_2}\sum_{2\leq m\leq n}\la \lambda_{g-2}\lambda_g |(2,1^{m-2})\ra_{g,m}^{\pi, \circ}\cdot\mathsf{Part}(n-m) \\
& &
-\frac{1}{24}\frac{t_1+t_2}{t_1t_2}\sum_{2\leq m\leq n}\la (2,1^{m-2}) \ra_{g,m}^{E\times \mathbb{C}^2, \circ}\cdot \widetilde{\mathsf{Part}}(n-m)\, .
\end{eqnarray*}
\end{prop}

The $g=1$ case of the connected/disconnected calculus takes a special form. For $n\geq 2$,
\begin{equation} \label{specialg1}
\la (2,1^{n-2})\ra_{1,n}^{\pi_{\mathbb{C}^2},\bullet} = 
-\frac{1}{24}\frac{t_1+t_2}{t_1t_2}\sum_{2\leq m\leq n}\la (2,1^{m-2}) \ra_{1,m}^{E\times \mathbb{C}^2, \circ}\cdot \widetilde{\mathsf{Part}}(n-m)\, .
\end{equation}
The genus $g=1$ case will be discussed in Section \ref{sec:con_discon_pf}. In fact, both sides
of \eqref{specialg1} vanish for $g=1$.

\subsection{Proof of Theorem \ref{Equiv1}}

Let $\mathcal{B}(u,Q)$ denote the power series
$$
\mathcal{B}(u,Q) = \sum_{g=1}^\infty \sum_{m=1}^\infty
\frac{|B_{2g-2}|}{(2g-2)!}(\sigma_{2g-1}(m)-\sigma_1(m)) Q^mu^{2g-3}.
$$
The $u^{-1}$ term of
$\mathcal{B}(u,y)$ corresponding to $g=1$
vanishes. Recall the definition 
$$\mathsf{Tr}_m = \frac{1}{t_1+t_2} \mathsf{trace} (\mathsf{M}_D^{\mathsf{Hilb}^m(\mathbb{C}^2)})
\, $$
of Section \ref{azazaz}.

 \begin{lem}\label{lem: trace}
    Under the variable change $-q=e^{iu}$, we have
    $$
    (-i)\sum_{m = 1}^\infty\mathsf{Tr}_m(q)\, Q^m = \mathcal{P}(Q)\mathcal{B}(u,Q).
    $$
\end{lem}
\begin{proof}
    The diagonal terms of $\mathsf{M}_D^{\mathsf{Hilb}^m(\mathbb{C}^2)}$ have been described in \cite[Section 2]{op}:
    $$
   \frac{1}{t_1+t_2} \mathsf{M}_D^{\mathsf{Hilb}^m(\mathbb{C}^2)} = \sum_{r=1}^\infty \left(\frac{r}{2}\frac{(-q)^r+1}{(-q)^r-1} - \frac{1}{2}\frac{(-q)+1}{(-q)-1}\right)\alpha_{-r}\alpha_r + \ldots ,
    $$
    where $\alpha_{-r}$ and $\alpha_r$
are the standard creation and annihilation
operators on the subspace $\mathcal{F}_m\subset \mathcal{F}$ of Fock space $$\mathcal{F}=\sum_{m=0}^\infty \mathcal{F}_m\, .$$
    Using the diagonal elements, we compute
\begin{equation}\label{eqn:trace_MD_expression}
    (-i)\sum_{m =1}^\infty \mathsf{Tr}_m(q)\, Q^m = \sum_{r \geq 1} \left(\frac{-ir}{2}\frac{(-q)^r+1}{(-q)^r-1} - \frac{-i}{2}\frac{(-q)+1}{(-q)-1}\right)\sum_{m = 1}^\infty \mathsf{trace}(\alpha_{-r}\alpha_r\,
    _{\mid \mathcal{F}_m}
    )\cdot Q^m\, .
    \end{equation}
    Since $\alpha_{-r}\alpha_r (|\mu\rangle) = r \cdot |\{i \mid \mu_i=r\}|\cdot|\mu\rangle$,
    \begin{equation}\label{eqn:energy}
    \sum_{m = 1}^\infty \mathsf{trace} (\alpha_{-r} \alpha_r\,_{\mid \mathcal{F}_m}) 
    \cdot Q^m = r\frac{Q^r}{1-Q^r}\prod_{m = 1}^\infty \frac{1}{1-Q^m}\, .
    \end{equation}
    The expansion of the cotangent function yields
    \begin{align}
    \begin{split}\label{eqn:cotangent}
    \frac{-ir}{2}\frac{(-q)^r +1}{(-q)^r-1} &= \frac{-ir}{2}\frac{e^{iur/2}+e^{-iur/2}}{e^{iur/2}-e^{iur/2}} \\ &= -\frac{r}{2}\cot\left(\frac{ur}{2}\right)\\
    &= -\frac{1}{u} + \sum_{h = 0}^\infty \frac{|B_{2h}|r^{2h}}{(2h)!}u^{2h-1}\, .
    \end{split}
    \end{align}
    After putting together \eqref{eqn:trace_MD_expression}, \eqref{eqn:energy} and \eqref{eqn:cotangent} and shifting $h =g-1$, we obtain:
    \begin{align*}
    \begin{split}
    (-i)\sum_{m =1}^\infty \mathsf{Tr}_m(q)\,  Q^m &= \mathcal{P}(Q) \left(\sum_{g=1}^\infty 
    \sum_{r=1}^\infty 
    \frac{|B_{2g-2}|}{(2g-2)!}(r^{2g-1}-r) \frac{Q^r}{1-Q^r} u^{2g-3}\right)\\
    &=\mathcal{P}(Q)\left( \sum_{g = 1}^\infty \frac{|B_{2g-2}|}{(2g-2)!}u^{2g-3}\sum_{r =1}^\infty(r^{2g-1}-r)(Q^r + Q^{2r}+\ldots) \right) \\
    &=\mathcal{P}(Q)\left( \sum_{g = 1}^\infty\frac{|B_{2g-2}|}{(2g-2)!}u^{2g-3}\sum_{m = 1}^\infty  Q^m \sum_{k \mid m}(k^{2g-1}-k) \right)\\
    &= \mathcal{P}(Q) \mathcal{B}(u,Q)\, .
    \end{split}
    \end{align*}
\end{proof}
By Theorem \ref{Equiv2} and Proposition \ref{prop:degeneration_form}, we obtain
\begin{equation} \label{xcce}
    \sum_{g=1}^\infty \sum_{n=1}^\infty u^{2g-3}Q^n\la \lambda_{g-2}\lambda_g |(2,1^{n-2})\ra_{g,n}^{\pi, \circ} = \frac{1}{24}\mathcal{B}(u,Q).
\end{equation}
The $g=1$ terms on the left are degenerate and interpreted as  
$$\la (2,1^{n-2})\ra_{1,n}^{\pi, \circ}=0\, ,$$
so all $u^{-1}$ vanish on both sides of \eqref{xcce}. 

By Lemma \ref{lem: trace} and \eqref{eqn:GWHilb_fixed}, the connected/disconnected equation for a fixed elliptic curve target \eqref{eqn:conn_disc_fixed} can be written as
\begin{equation*}
    \mathcal{P}(Q)\left(\sum_{g=1}^\infty \sum_{n = 1}^\infty \la(2, 1^{n-2})\ra_{g,n}^{E \times \CC^2, \circ}u^{2g-3}Q^n\right) = i(t_1 + t_2)\sum_{n = 1}^\infty \mathsf{Tr}_n(q)\,  Q^n = -(t_1 + t_2)\mathcal{P}(Q)\mathcal{B}(u,Q)\ ,
\end{equation*}
after $-q=e^{iu}$. The connected/disconnected calculus of
Proposition \ref{prop:con_discon} then yields
\begin{align*}
    \begin{split}
    \sum_{g= 1}^\infty \sum_{n=1}^\infty 
    \langle (2,1^{n-2})\rangle_{g,n}^{\pi_{\mathbb C^2}, \bullet}Q^n u^{2g-3} &= -\frac{1}{24}\frac{t_1 + t_2}{t_1t_2} (-(t_1 + t_2)\mathcal{B}(u,Q))\cdot \widetilde{\mathcal{P}}(Q) + \frac{(t_1 + t_2)^2}{t_1t_2} \frac{1}{24}\mathcal{B}(u,Q)\mathcal{P}(Q)\\
    &= \frac{1}{24}\frac{(t_1 + t_2)^2}{t_1t_2} \mathcal{B}(u,Q) \mathcal{P}(Q)(1 + \log \mathcal{P}(Q))  \\
    &= \frac{1}{i}\frac{1}{24}\frac{(t_1 + t_2)^2}{t_1t_2} \left(\sum_{n=1}^\infty \mathsf{Tr}_n(q) \, Q^n\right)\left(1 + \sum_{k = 1}^\infty \sigma_{-1}(k)Q^k\right)\, .
    \end{split}\, 
\end{align*}
In the second equality, we have applied Lemma \ref{lem:Ptilde_P}.
In the third equality, we have used
$$
\operatorname{log} \mathcal{P}(Q) = -\sum_{n =1}^\infty \operatorname{log}(1-Q^n) = \sum_{k =1}^\infty \sigma_{-1}(k)Q^k\, .
$$
Since, by Theorem A, 
\begin{equation*}
    \langle D\rangle_1^{\Hilb} = \frac{1}{i}\langle (2,1^{n-2})\rangle^{\pi_{\mathbb C^2}, \bullet}\, ,
\end{equation*}
 $\langle D\rangle_1^{\Hilb}$ 
is $\frac{-1}{24}\frac{(t_1 + t_2)^2}{t_1t_2}$ times the $Q^n$ coefficient of
$$
\left(\sum_{n =1}^\infty \mathsf{Tr}_n(q)\,  Q^n\right)\left(1 + \sum_{k = 1}^\infty \sigma_{-1}(k)Q^k\right)\, .
$$
We have proven Theorem \ref{Equiv1}. \qed

\subsection{Proof of Proposition \ref{prop:con_discon}}\label{sec:con_discon_pf}
We study here the connected/disconnected calculus for the family $\pi_{\mathbb{C}^2}$.
To start, we consider the universal elliptic curve
$$ \pi:\mathcal{E} \rightarrow \overline{\mathcal{M}}_{1,1}\, $$
with section $\mathsf{p}_1$.
Let $(m, \mathbf k, \mathbf g)$ be a triple,
\begin{itemize}
    \item $2 \leq m \leq n$,
    \item $\mathbf k = (k_1, \ldots , k_s)$ is a partition of $n-m$, 
    \item $\mathbf g = (g_0, \ldots , g_s)$ is a partition of $g+s-1$,
\end{itemize}
where the parts of $\mathbf{k}$ and $\mathbf{g}$ are positive integers.
We define
\begin{align*}
    \Mbar_{m,\mathbf k , \mathbf g}^{\mathsf{fiber}} &= \Mbar_{g_0}^\circ(\pi,(2,1^{m-2})) \times_{\Mbar_{1,1}} \Mbar_{g_1}^\circ(\pi,(1^{k_1})) \times_{\Mbar_{1,1}}\ldots \times_{\Mbar_{1,1}} 
    \Mbar_{g_s}^\circ(\pi,(1^{k_s})),\\
    \Mbar_{m,\mathbf k , \mathbf g} &= \Mbar_{g_0}^\circ(\pi,(2,1^{m-2})) \times 
    \Mbar^\circ_{g_1}(\pi,(1^{k_1})) \times\ldots \times \Mbar_{g_s}^\circ(\pi,(1^{k_s}))\, ,
\end{align*}
where 
$\Mbar_{m,\mathbf k , \mathbf g}^{\mathsf{fiber}}$ is the base change of $\Mbar_{m,\mathbf k , \mathbf g}$ along the diagonal $\Delta : \Mbar_{1,1} \to \Mbar_{1,1}\times \ldots \times \Mbar_{1,1}$.

The fixed locus of $\Mbar^{\bullet}_{g}(\pi_{\mathbb C^2},n, (2,1^{n-2}))$ for the action of the torus $\mathsf{T}=(\mathbb C^*)^2$ is the disjoint union of the stacks $\Mbar_{m, \mathbf k , \mathbf g}^{\mathsf{fiber}}$, up to an automorphism factor which we denote by $\operatorname{Aut}(\mathbf k, \mathbf g)$, and the restriction of the virtual class $[\Mbar_{g}^\bullet(\pi_{\CC^2}, (2,1^{n-2}))]^{vir}$ to  $\Mbar_{m, \mathbf k , \mathbf g}^{\mathsf{fiber}}$ agrees with $\Delta^![\Mbar_{m,\mathbf k, \mathbf g}]^{vir}$. By the localization formula of \cite{GP}, we have 
\begin{equation*}
    [\Mbar^\bullet_{g}(\pi_{\CC^2}, (2,1^{n-2}))]^{vir} = \sum_{(m,\mathbf k, \mathbf g )} \frac{1}{\operatorname{Aut}(\mathbf k, \mathbf g)}\frac{\Delta^![\Mbar_{m,\mathbf k, \mathbf g}]^{vir}}{e_{(\mathbb{C}^*)^2}(N^{vir})}\, ,
\end{equation*}
where the contribution of the virtual normal bundle on each factor of $\Mbar_{m,\mathbf k, \mathbf g}$ is given by
\begin{align}\label{ddcc4}
\begin{split}
    \frac{1}{\mathsf{e}_{\mathsf{T}}(N^{vir})} &=\frac{c(\mathbb E^\vee\otimes t_1)c(\mathbb E^\vee\otimes t_2)}{t_1t_2} \\
    &= -\frac{t_1+t_2}{t_1t_2}\lambda_{g_i}\lambda_{g_i-1} + \frac{(t_1+t_2)^2}{t_1t_2}\lambda_{g_i}\lambda_{g_i-2}-(t_1+t_2)\lambda_{g_i-1}\lambda_{g_i-2}+\ldots
\end{split}
\end{align}
if $g_i >1$, and $-\frac{t_1+t_2}{t_1t_2}\lambda_{1} + 1$ if $g_i=1$. We denote the class
\eqref{ddcc4}
by $N_{g_i}$, and note that the lower order terms in the Hodge classes
will not contribute for dimensional reasons. After integration, we obtain 
\begin{align}\label{eqn:first_localization}
    \begin{split}
        \langle (2,1^{n-2})\rangle_{g,n}^{\pi_{\mathbb C^2}, \bullet} &= \sum_{(m,\mathbf k , \mathbf g)}\frac{1}{\operatorname{Aut}(\mathbf k, \mathbf g)}\int_{[\Mbar^{\mathsf{fiber}}_{(m,\mathbf k , \mathbf g)}]^{vir}}\prod_{i=0}^sN_{g_i}\\
        &=\sum_{(m,\mathbf k , \mathbf g)}\frac{1}{\operatorname{Aut}(\mathbf k, \mathbf g)}\sum_{j = 0}^s\int_{[\Mbar_{(m,\mathbf k , \mathbf g)}]^{vir}}\prod_{i=0}^sN_{g_i} \prod_{i \neq j}\operatorname{ev}_{i}^*([\mathsf{pt}])\\
    &=\sum_{(m,\mathbf k , \mathbf g)}\frac{1}{\operatorname{Aut}(\mathbf k, \mathbf g)}\left(\langle N_{g_0}|(2,1^{m-2})\rangle_{g_0,m}^{\pi, \circ} \prod_{i= 1}^s\langle N_{g_i} |(1^{k_i})\rangle_{g_i, k_i}^{E,\circ }  \right.\\
    & + \left.\langle N_{g_0}|(2,1^{m-2})\rangle_{g_0,m}^{E, \circ}\sum_{j = 1}^s-\frac{t_1+t_2}{t_1t_2}\langle \lambda_{g_j}\lambda_{g_j-1}|(1^{k_j})\rangle_{g_j, k_j}^{\pi, \circ}\prod_{i \neq j}\langle N_{g_i} |(1)\rangle_{g_i, k_i}^{E,\circ }\right)
    \end{split}
\end{align}
where $\operatorname{ev}_i : \Mbar_{(m,\mathbf k , \mathbf g)} \to \Mbar_{1,1}$ recovers the isomorphism class of the target curve of the $i$-th map. The second equality follows from the decomposition of the diagonal of $\Mbar_{1,1}$:
$$
[\Delta] = \sum_{j=0}^s[\mathsf{pt}] \boxtimes \ldots \boxtimes [\mathsf{pt}]\boxtimes 1 \boxtimes [\mathsf{pt}]\boxtimes \ldots \boxtimes[\mathsf{pt}] \in  \operatorname{\mathsf{CH}}^{s}(\Mbar_{1,1}\times \cdots \times \Mbar_{1,1})\, .
$$

If $g_i>1$,
the class $N_{g_i}$ contains a $\lambda_{g_i}$ factor. Hence, $\langle N_{g_i}|(1^{k_i})\rangle_{g_i, k_i}^{E, \circ}=0$ unless $g_i=1$, in which case 
$$\langle N_{1}|(1^{k_i})\rangle_{1, k_i}^{E, \circ}
=\langle (1^{k_i})\rangle_{1,k_i}^{E, \circ}
= \frac{\sigma(k_i)}{k_i}\, .$$
\begin{lem}\label{lem:1^d_evaluation}
   If $g_i>1$,  we have $\langle \lambda_{g_i} \lambda_{g_i-1}|(1^{k_i})\rangle_{g_i,k_i}^{\pi,\circ} =0$. If $g_i=1$,
   $$\langle \lambda_{1} \lambda_{0}|(1^{k_i})\rangle_{1,k_i}^{\pi,\circ}
   =\frac{1}{24}\langle (1^{k_i})\rangle_{1,k_i}^{E, \circ} = \frac{\sigma_1(k_i)}{24k_i}\, .$$
\end{lem}
\begin{proof}
The first step of the argument is the
equality
$$\langle \lambda_{g_i} \lambda_{g_i-1}|(1^{k_i})\rangle_{g_i,k_i}^{\pi,\circ} =
\langle \lambda_{g_i} \lambda_{g_i-1} \rangle_{g_i,k_i}^{\pi,\circ}$$
obtained from the standard degeneration to the normal cone of
the section $\mathsf{p}_1$.

We now follow the geometric notation and analysis of Section \ref{sec:virtual_class_comparison}.
The moduli space of stable maps to the fibers
of 
$\pi$
lies over $\overline{\mathcal{M}}_{1,1}$,
$$\epsilon: \Mbar^\circ_{g_i}(\pi,k_i) \rightarrow 
\overline{\mathcal{M}}_{1,1}
\, .$$
The universal curve $\mu: \mathcal{C}
\rightarrow \overline{\mathcal M}^\circ_{g_i}(\pi,k_i)$
carries a universal evaluation map 
\[
\begin{tikzpicture}[scale=2, every node/.style={font=\small}]
  \node (C) at (0,1.3) {$\mathcal{C}$};
  \node (T) at (1.5,1.3) {$\mathcal{T}$};
  \node (M) at (0.75,0) {$\overline{\mathcal{M}}^\circ_{g_i}(\pi,k_i)$};

  \draw[->] (C) -- (T) node[midway, above] {$f$};
  \draw[->] (C) -- (M) node[midway, left] {$\mu$};
  \draw[->] (T) -- (M) node[midway, right] {$\nu$};
\end{tikzpicture}
\]
to the universally expanded target
\[
\begin{tikzpicture}[scale=2, every node/.style={font=\small}]
  \node (C) at (0,1.3) {$\mathcal{T}$};
  \node (T) at (1.5,1.3) {$\overline{\mathcal M}^\circ_{g_i}(\pi,k_i)\times _
  {\overline{\mathcal{M}}_{1,1}} \mathcal{E}$};
  \node (M) at (0.75,0) {$\overline{\mathcal{M}}^\circ_{g_i}(\pi,k_i)\, .$};

  \draw[->] (C) -- (T) node[midway, above] {$h$};
  \draw[->] (C) -- (M) node[midway, left] {$\nu$};
  \draw[->] (T) -- (M) node[midway, right] {$\pi_M$};
\end{tikzpicture}
\]
The target $\mathcal{T}$ is a family of elliptic
curves over $\overline{\mathcal M}^\circ_{g_i}(\pi,k_i)$
with possible expansion over the nodal fibers. 
We have a pull-back map{\footnote{$\mathbb{E}_1$ is
the Hodge bundle on ${\overline{\mathcal{M}}}_{1,1}$.}}
\begin{equation} \label{injinj2}
\epsilon^*\mathbb{E}_1 \cong \nu_*\omega_{\mu} \ \stackrel{f^*}{\rightarrow}\  \mu_*\omega_{\mu} \cong \mathbb{E}_g
\end{equation}
over $\overline{\mathcal M}_{g_i}(\pi,k_i)$.
Since $k_i>0$, the map \eqref{injinj2} is injective,
so we obtain an exact sequence 
\begin{equation*} \label{vrrf2}
0 \rightarrow
\epsilon^* \mathbb{E}_1 \rightarrow
\mathbb{E}_g
\rightarrow
\mathbb{F} \rightarrow 0\, ,
\end{equation*}
where $\mathbb{F}$ is a rank $g-1$ vector bundle on
$\overline{\mathcal M}^\circ_{g_i}(\pi,k_i)$.
We therefore have a factorization
\begin{equation}
\lambda_g  \ = \ c_1(\epsilon^*\mathbb{E}_1) \cdot c_{g-1}(\mathbb{F})  
\ = \ \frac{1}{24} \epsilon^*([E]) \cdot \lambda_{g-1}\, 
\label{faccc2}
\end{equation}
on 
$\overline{\mathcal M}_{g_i}(\pi,k_i)$, where $[E]\in \overline{\mathcal{M}}_{1,1}$
is the moduli point of a fixed nonsingular elliptic curve $E$. 

The factorization \eqref{faccc2} of $\lambda_g$
implies 
\begin{equation}\label{vrr44}
\int_{[\overline{\mathcal M}^\circ_{g_i}(\pi,k_i)]^{vir}}\lambda_{g_i}\lambda_{g_i-1} =
\frac{1}{24} 
\int_{[\overline{\mathcal M}^\circ_{g_i}(E,k_i)]^{vir}}\lambda^2_{g_i-1} \, .
\end{equation}
If $g_i>1$, then $\lambda_{g_i-1}^2 = 2\lambda_{g_i}\lambda_{g_i-2}$. Hence, the integral
\eqref{vrr44}  is 0 by $\lambda_g$-vanishing for elliptic targets.
If $g_i=1$, the integral is easily evaluated as claimed.
\end{proof}

Therefore, the only possible non-zero contributions to the right hand side of \eqref{eqn:first_localization} come from triplets $(e,\mathbf k , \mathbf g)$ where $\mathbf g = (g,1,\ldots , 1)$
with $g\geq 1$.

\noindent $\bullet$
For $g>1$, we have
\begin{equation*}
\langle N_{g}|(2, 1^{m-2})\rangle_{g,m}^{\pi, \circ} =\frac{(t_1+t_2)^2}{t_1t_2}\langle \lambda_{g}\lambda_{g-2}|(2,1^{m-2})\rangle_{g,m}^{\pi, \circ} \, .
\end{equation*}
By a similar localization analysis, we have
\begin{equation*}
\langle N_{g}|(2,1^{m-2})\rangle_{g,m}^{E, \circ} =-(t_1+t_2)\langle \lambda_{g-1}\lambda_{g-2}|(2,1^{m-2})\rangle_{g,m}^{E, \circ} = \langle (2,1^{m-2})\rangle^{E \times \mathbb C^2}_{g,m}
\end{equation*}
Therefore, the formula in \eqref{eqn:first_localization} simplifies to
\begin{align*}
    \begin{split}
        \langle (2,1^{n-2})\rangle_{g,n}^{\pi_{\mathbb C^2}, \bullet}
    =\sum_{m=2}^n\sum_{\mathbf k \,\vdash n-m}&\left(\frac{(t_1+t_2)^2}{t_1t_2}\langle \lambda_{g}\lambda_{g-2}|(2,1^{m-2})\rangle_{g,m}^{\pi, \circ} \frac{\prod_{i} \langle (1^{k_i})\rangle_{1,k_i}^{E,\circ}}{\operatorname{Aut}(\mathbf k)} \right. \\
    &\phantom{+}\left.
    -\frac{1}{24}\frac{t_1+t_2}{t_1t_2}\langle (2,1^{m-2})\rangle_{g,m}^{E\times\mathbb C^2,\circ} \frac{l(\mathbf k)\prod_{i} \langle (1^{k_i})\rangle_{1,k_i}^{E,\circ}}{\operatorname{Aut}(\mathbf k)}\right).
    \end{split}
\end{align*}
Summing over all possible partitions, we obtain the coefficients $\operatorname{\mathsf{Part}}(n-m)$ and $\widetilde{\operatorname{\mathsf{Part}}}(n-m)$ 
of Proposition \ref{prop:con_discon}
as defined in Section \ref{sec:conn_disconn_calc}.

\noindent $\bullet$
For $g=1$, $\langle N_{1}|(2,1^{m-2})\rangle_{1,m}^{\pi, \circ} = \langle (2,1^{m-2})\rangle_{1,m}^{\pi, \circ}  =0$ by Proposition \ref{prop:degeneration_form}, so both sides of formula \eqref{specialg1} after Proposition \ref{prop:con_discon} vanish.
\qed


\subsection{Proof of Proposition \ref{prop:degeneration_form}}\label{sec:degeneration_pf}


Let  $g \geq 2$.
Let $\operatorname{Bl_{\mathsf{p}_1 \times \{0\}} (\mathcal E \times \mathbb A^1)}$ be the degeneration to the normal cone
of the section $\mathsf{p}_1$. The special fiber over $0 \in \mathbb A^1$ is the family
$$
\pi \cup \pi_P: \mathcal E \cup \mathbb P (T_\mathsf{p_1} \mathcal E \oplus \mathbb C) \longrightarrow \Mbar_{1,1},
$$
where the gluing identifies the section $\mathsf{p}_1$ of $\mathcal{E}$ with the $0$-section of 
$\mathbb P (T_{\mathsf{p}_1} \mathcal E \oplus \mathbb C)$. 
The universal section $\mathsf{p}_1$ is now the $\infty$-section of $\pi_P$. 
We will use the degeneration formula
\begin{equation}\label{eqn:general_deg_formula}
    \langle\tau_1(\mathsf{p}_1)\lambda_g \lambda_{g-2}\rangle_{g,n}^{\pi, \circ}=\sum_{(\Gamma_1, \Gamma_2, \mu)}\langle \Lambda_1 | \mu\rangle_{\Gamma_1}^{\pi,\bullet}\langle \Lambda_2\tau_1(\infty) | \mu^*\rangle_{\Gamma_2}^{\pi_P,\bullet}\, ,
\end{equation}
with notation:
\begin{enumerate}
\item[$\bullet$] $\Gamma_1$ and $\Gamma_2$ are possibly disconnected topological types corresponding to
the domains mapping to the components $\mathcal{E}$ and $\mathbb P (T_\mathsf{p} \mathcal E \oplus \mathbb C)$,
\item[$\bullet$]
$\mu$ is a partition of $d$ decorated with elements of a basis of $H^*(\mathsf{p}) = H^*(\Mbar_{1,1})$, 
\item[$\bullet$] $\Lambda_1$ and $\Lambda_2$ reflect the distribution of the Hodge insertions.
\end{enumerate}

The Hodge insertion $\lambda_g\lambda_{g-2} = 2\lambda_{g-1}^2$ annihilates all contributions of graphs which are not of compact type or have more than two vertices with positive genus. If there are exactly
two vertices of genera $0<g_1, g_2<g$, then the Hodge insertion distributes as
\begin{equation}
    \lambda_g \lambda_{g-2}|_{\Mbar_{g_1} \times \Mbar_{g_2}} = \lambda_{g_1}\lambda_{g_1-1} \boxtimes \lambda_{g_2}\lambda_{g_2-1}.
\end{equation}

An analysis of the virtual dimension of each vertex shows that all vertex contributions vanish for dimension reasons
except for the following list:
\begin{align}
\begin{split}\label{eqn:vertex1}
    \langle \lambda_g \lambda_{g-2}|(2[\Mbar_{1,1}],1[\Mbar_{1,1}],\ldots ,1[\Mbar_{1,1}])\rangle_{g}^{\pi,\circ}
\end{split}\\
\begin{split}\label{eqn:vertex2}
    \langle \lambda_g \lambda_{g-2}|(1[\mathsf{pt}],1[\Mbar_{1,1}],\ldots ,1[\Mbar_{1,1}])\rangle_{g}^{\pi,\circ}
\end{split} \\
\begin{split}\label{eqn:vertex3}
    \langle \lambda_{g_1} \lambda_{g_1-1}|(1[\Mbar_{1,1}],\ldots ,1[\Mbar_{1,1}])\rangle_{g_1}^{\pi,\circ}
\end{split}\\
\begin{split}\label{eqn:vertex4}
    \langle (1[\mathsf{pt}])\rangle_{0}^{\pi_P, \circ}
\end{split}\\
\begin{split}\label{eqn:vertex5}
    \langle \tau_1(\infty)|(1[\Mbar_{1,1}])\rangle_{0}^{\pi_P, \circ}
\end{split}\\
\begin{split}\label{eqn:vertex6}
    \langle \tau_1(\infty)|(2[\mathsf{pt}])\rangle_{0}^{\pi_P, \circ}
\end{split}\\
\begin{split}\label{eqn:vertex7}
    \langle \tau_1(\infty)|(1[\mathsf{pt}],1[\mathsf{pt}])\rangle_{0}^{\pi_P, \circ}
\end{split}\\
\begin{split}\label{eqn:vertex8}
    \langle \tau_1(\infty)\lambda_{g_2}\lambda_{g_2-1}|(1[\mathsf{pt}])\rangle_{g_2}^{\pi_P, \circ}
\end{split}
\end{align}
Three further vanishings hold for other reasons:
\begin{enumerate}
\item[$\bullet$] Vertex \eqref{eqn:vertex2} vanishes because
\begin{equation*}
    [\Mbar_g(\pi/\mathsf{p}_1)]^{vir} \cap \operatorname{ev}_{\Mbar_{1,1}}^*(\mathsf{pt}) = [\Mbar_g(E/p_1)]^{vir}
\end{equation*}
pairs to zero with $\lambda_g$. 
\item[$\bullet$]
Vertex \eqref{eqn:vertex7} vanishes because of the two point
conditions over $\Mbar_{1,1}$. 
\item[$\bullet$]
Vertex \eqref{eqn:vertex3} vanishes unless $g_1=1$ by  Lemma \ref{lem:1^d_evaluation}. 
\end{enumerate}

It follows that the only possible combinatorial types that contribute to the right hand side of \eqref{eqn:general_deg_formula} are given by the two configurations of Figure \ref{fig:deg_contributions}, where the first configuration is counted $n$ times.

\begin{figure}[ht]
    \centering
    \begin{tikzpicture}[scale=0.7]
        \node[style={draw, circle, minimum size=0.8cm, font=\small}] (A) at (0,0) {$1$};
        \node (B1) at (2,6) {};
        \node (B2) at (2,3) {};
        \node (B3) at (2,0) {};
        \node (B4) at (2,-2) {};

        \node (C1) at (3.2,6) {$1[{\Mbar_{1,1}}]$};
        \node (C2) at (3.2,3) {$1[\Mbar_{1,1}]$};
        \node (C3) at (3.2,1) {$\vdots$};
        \node (C4) at (3.2,-2) {$1[{\Mbar_{1,1}}]$};

        \node (D1) at (5.7,6) {$1[\mathsf{pt}]$};
        \node (D2) at (5.7,3) {$1[\mathsf{pt}]$};
        \node (D3) at (5.7,1) {$\vdots$};
        \node (D4) at (5.7,-2) {$1[\mathsf{pt}]$};

        \node[style={draw, circle, minimum size=0.8cm, font=\tiny}] (E1) at (8,6) {$g-1$};
        \node[style={draw, circle, minimum size=0.8cm, font=\small}] (E2) at (8,3) {$0$};
        \node (E3) at (8,1) {$\vdots$};
        \node[style={draw, circle, minimum size=0.8cm, font=\small}] (E4) at (8,-2) {$0$};

        \node (F1) at (9.5,6) {};
        
        \draw (A) -- (B1);
        \draw (A) -- (B2);
        \draw (A) -- (B4);
        
        \draw (C1) -- (D1);
        \draw (C2) -- (D2);
        \draw (C4) -- (D4);

        \draw (D1) -- (E1);
        \draw (D2) -- (E2);
        \draw (D4) -- (E4);
        
        \draw (E1) -- (F1);
    \end{tikzpicture}
     \hspace{5mm}
    \begin{tikzpicture}[scale=0.7]
        \node[style={draw, circle, minimum size=0.8cm, font=\small}] (A) at (0,0) {$g$};
        \node (B1) at (2,6) {};
        \node (B2) at (2,3) {};
        \node (B3) at (2,0) {};
        \node (B4) at (2,-2) {};

        \node (C1) at (3.2,6) {$2[{\Mbar_{1,1}}]$};
        \node (C2) at (3.2,3) {$1[\Mbar_{1,1}]$};
        \node (C3) at (3.2,1) {$\vdots$};
        \node (C4) at (3.2,-2) {$1[{\Mbar_{1,1}}]$};

        \node (D1) at (5.7,6) {$2[\mathsf{pt}]$};
        \node (D2) at (5.7,3) {$1[\mathsf{pt}]$};
        \node (D3) at (5.7,1) {$\vdots$};
        \node (D4) at (5.7,-2) {$1[\mathsf{pt}]$};

        \node[style={draw, circle, minimum size=0.8cm, font=\small}] (E1) at (8,6) {$0$};
        \node[style={draw, circle, minimum size=0.8cm, font=\small}] (E2) at (8,3) {$0$};
        \node (E3) at (8,1) {$\vdots$};
        \node[style={draw, circle, minimum size=0.8cm, font=\small}] (E4) at (8,-2) {$0$};

        \node (F1) at (9.3,6) {};
        
        \draw (A) -- (B1);
        \draw (A) -- (B2);
        \draw (A) -- (B4);
        
        \draw (C1) -- (D1);
        \draw (C2) -- (D2);
        \draw (C4) -- (D4);

        \draw (D1) -- (E1);
        \draw (D2) -- (E2);
        \draw (D4) -- (E4);
        
        \draw (E1) -- (F1);
    \end{tikzpicture}

    \caption{Non-zero contributions to \eqref{eqn:general_deg_formula}}\label{fig:deg_contributions}
\end{figure}
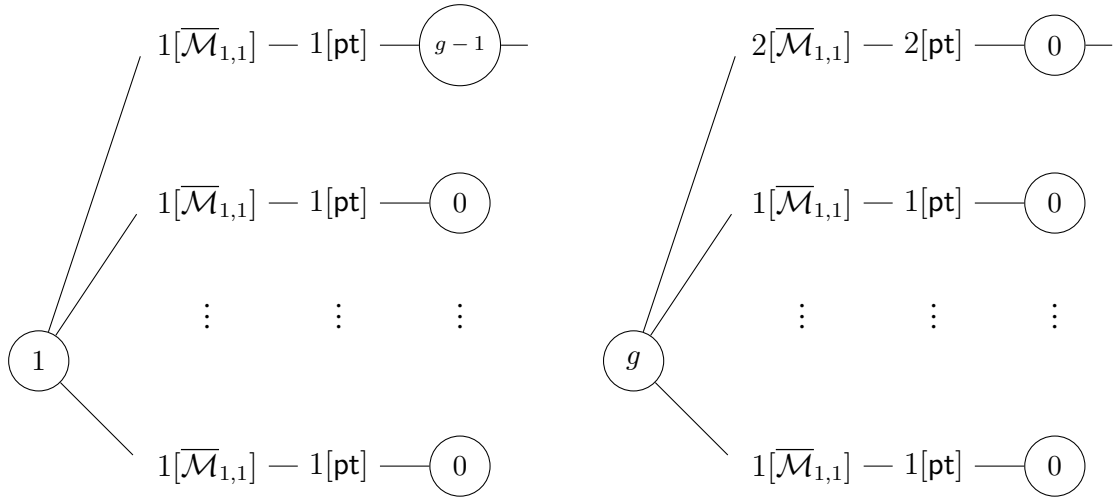
\noindent The vertices \eqref{eqn:vertex4} and \eqref{eqn:vertex6} contribution factors $1$ and $1/2$ respectively (the latter
is cancelled by the multiplicity which occurs in  the degeneration formula).

The above vertex analysis show that \eqref{eqn:general_deg_formula} specializes to:
$$
\langle \tau_1(\mathsf{p}_1)\lambda_g \lambda_{g-2}\rangle_{g,n}^{\pi, \circ} = 
\frac{\sigma_1(n)}{24} \langle \tau_1(\infty)\lambda_{g-1}\lambda_{g-2}|(1[\mathsf{pt}])\rangle_{g-1,1}^{\pi_P} +
\langle \lambda_g \lambda_{g-2}|(2,1^{n-2})\rangle_{g,n}^{\pi,\circ} 
\, .
$$
To complete the proof of Proposition \ref{prop:degeneration_form}, we must
evaluate
$$
\langle \tau_1(\infty)\lambda_{g-1}\lambda_{g-2}|(1[\mathsf{pt}])\rangle_{g-1,1}^{\pi_P} = \langle \tau_1(\infty)\lambda_{g-1}\lambda_{g-2}\rangle_{g-1,1}^{\mathbb P^1}\, ,
$$
where the equality with the absolute invariant is proven by the standard degeneration method.
To calculate, we localize with respect to the $\mathbb{C}^*$- action on $\mathbb P^1$. Since the integrand has a $\lambda_g \lambda_{g-1}$ insertion, there are only two components of the fixed locus that contribute, corresponding to the graphs in Figure \ref{fig:localization_P1},
see \cite{GP}.
\begin{figure}[ht]
    \centering
    \begin{tikzpicture}[scale = 0.7]
        \node[style={draw, circle, minimum size=0.8cm, font=\small}] (A) at (0,0) {0};
        \node[style={draw, circle, minimum size=0.8cm, font=\tiny}] (B) at (3,0) {$g-1$};
        
        \draw (A) -- (B);

        \draw[gray, thick] (B) -- ++(45:1.2);
    \end{tikzpicture}
    \hspace{2cm}
    \begin{tikzpicture}[scale = 0.7, every node/.style={draw, circle, minimum size=0.8cm, font=\small}]
        \node[style={draw, circle, minimum size=0.8cm, font=\tiny}] (A) at (0,0) {$g-1$};
        \node[style={draw, circle, minimum size=0.8cm, font=\small}] (B) at (3,0) {0};
        
        \draw (A) -- (B);

        \draw[gray, thick] (B) -- ++(45:1);
    \end{tikzpicture}
    \caption{Localization contributions, where half edges correspond to marked points, vertices are contracted components, and edges are non-contracted components.}\label{fig:localization_P1}
\end{figure}
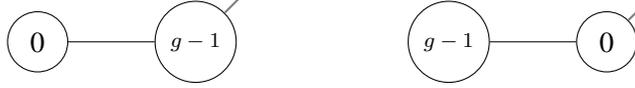

\noindent After expanding the localization formula, we conclude
\begin{equation*}
    \langle \tau_1(\infty)\lambda_{g-1}\lambda_{g-2}\rangle_{g-1,1}^{\mathbb P^1} = (2g-2) \int_{\Mbar_{g-1,1}} \frac{\lambda_{g-1}\lambda_{g-2}c(\mathbb E^\vee)}{1-\psi} = \frac{|B_{2g-2}|}{(2g-2)!}\, ,
\end{equation*}
where the Hodge integral on the right is calculated in \cite{P2}. \qed

\section{Reconstruction of multi-point invariants}\label{sec:reconstruction}

\subsection{Reduction to $1$-point series }

We recall the notation introduced in Section \ref{sechs}:
\begin{equation}\label{cvfre}
\blang D^{* k_1},...,D^{*k_\ell}\rangle_{1}^{\hilbnc} = \sum_{d=0}^\infty
\blang D^{* k_1},...,D^{*k_\ell}\rangle_{1,d}^{\hilbnc} q^d\, .
\end{equation}
In order to emphasize the number of insertions, we also write
$$
\blang D^{* k_1},...,D^{*k_\ell}\rangle_{1}^{\hilbnc}
= \blang D^{* k_1},...,D^{*k_\ell}\rangle^{\hilbnc}
_{1,\, {\fbox{{\tiny $\ell$}}}}\, ,
$$
where the boxed subscript indicates insertion number (not the curve degree as in \eqref{cvfre}).

We prove here the reduction of the genus $1$ Gromov-Witten theory of
$\hilbnc$ to $1$-point series as stated in Theorem \ref{Red1}:

\noindent {\em To every genus $1$ series $\blang D^{\star k_1},...,D^{\star k_\ell}\brang^{\hilbnc}_{1}$, there are canonically
associated functions
\begin{equation*}
\big\{ C_{k,m} \big\}_{0\leq k \leq |\mathsf{Part}(n)|-1,\,  0\leq m\leq \ell-1} \subset  \mathbb{Q}(t_1,t_2,q)    
\end{equation*}
 for which the following equation holds:
\begin{equation*}
\blang D^{* k_1},...,D^{*k_\ell}\rangle_{1}^{\hilbnc}=\sum_{k=0}^{|\mathsf{Part}(n)|-1}\, \sum_{m=0}^{\ell-1} \, C_{k,m}\cdot \left(q\frac{d}{dq}\right)^m\blang D^{*k}\brang_{1}^{\hilbnc}\, .    \end{equation*}}

\subsection{Proof of Theorem \ref{Red1}}\label{sec:proof_reconst}

We adapt the reconstruction strategy of \cite[Section 6]{w} to prove the result.
Let $m\geq 0$, and fix classes $$v_1,v_2,v_3,v_4, \phi_1,...,\phi_m\in H^*_\T(\hilbnc)\, .$$
Pulling back Getzler's relation from $\overline{\mathcal M}_{1,4}$ to $\overline{\mathcal M}_{1,4+m}(\hilbnc,d)$ and summing over the curve degree $d$, we obtain the following equation (see \cite[Equation (27)]{Lin})
\begin{equation*}
3\sum_{h\in S_4}\langle v_{h(1)}\star v_{h(2)}, v_{h(3)}\star v_{h(4)},\phi_1,...,\phi_m\rangle^{\hilbnc}_{1,\,\fbox{\tiny{2+m}}}
\, \sim\, 
4\sum_{h\in S_4}\langle v_{h(1)}\star v_{h(2)}\star v_{h(3)},v_{h(4)}, \phi_1,...,\phi_m\rangle^{\hilbnc}_{1,\,{\fbox{\tiny{2+m}}}}\, ,
\end{equation*}
where the relation $\mathsf{X}\sim\mathsf{Y}$ means that the difference $\mathsf{X}-\mathsf{Y}$ is an explicit sum of arbitrary genus $0$ invariants and  {\em genus $1$ invariants with at most $m+1$ insertions}.
Getzler's relation can be equivalently cast in the following form:
\begin{eqnarray*}
\Psi(v_1,v_2,v_3,v_4) & = &
\ \ \, \langle v_1\star v_2,v_3\star v_4,\phi_1,...,\phi_m\rangle^{\hilbnc}_{1,\,\fbox{\tiny{2+m}}}
+\langle v_1\star v_3,v_2\star v_4,\phi_1,...,\phi_m\rangle^{\hilbnc}_{1,\,\fbox{\tiny{2+m}}}\\
& &+\langle v_1\star v_4,v_2\star v_4,\phi_1,...,\phi_m\rangle^{\hilbnc}_{1,\,\fbox{\tiny{2+m}}}\\
& &-\langle v_1,v_2\star v_3\star v_4,\phi_1,...,\phi_m\rangle^{\hilbnc}_{1,\,\fbox{\tiny{2+m}}}
-\langle v_2,v_1\star v_3\star v_4,\phi_1,...,\phi_m\rangle^{\hilbnc}_{1,\,\fbox{\tiny{2+m}}}\\
& & -\langle v_3,v_1\star v_2\star v_4,\phi_1,...,\phi_m\rangle^{\hilbnc}_{1,\,\fbox{\tiny{2+m}}}
-\langle v_4,v_1\star v_2\star v_3,\phi_1,...,\phi_m\rangle^{\hilbnc}_{1,\,\fbox{\tiny{2+m}}}  
\end{eqnarray*}
satisfies the relation 
\begin{equation}\label{eqn:getzler_symbol}
\Psi(v_1,v_2,v_3,v_4)\sim 0\, .   
\end{equation}


Let
 $a, b\in H^{\leq 2}_\T(\hilbnc)$, and let
$\phi_1,...,\phi_m \in H^*_\T(\hilbnc)$. For integers $l\geq 2$ and  $0\leq i \leq l$ define the following series:
\begin{eqnarray*}
h(l)&=&\langle a\star b\star D^{*(l-2)}, D^{*2}, \phi_1,...,\phi_m\rangle^{\hilbnc}_{1,\, \fbox{\tiny{2+m}}}\, ,\\   
f(i)&=&\langle a\star D^{*i}, b\star D^{*(l-i)}, \phi_1,...,\phi_m\rangle^{\hilbnc}_{1,\, \fbox{\tiny{2+m}}}\, .
\end{eqnarray*}
From the string and divisor equations, we immediately obtain 
\begin{equation*}
f(0)\sim 0\, ,\quad f(l)\sim 0\, .    
\end{equation*}
By expanding the definitions, we see, for $0\leq k\leq l-2$,
\begin{eqnarray}\label{eqn:difference_inital}
 -\Psi(a\star D^{*k}, b\star D^{*(l-2-k)}, D,D) & = &
f(k+2)-2f(k+1)+f(k)-h(l) \\ \nonumber
& &\ \ \  +2\langle D,a\star b\star D^{*(l-1)},\phi_1,...,\phi_m\rangle^{\hilbnc}_{1,\, \fbox{\tiny{2+m}}}\, .
\end{eqnarray}
By applying \eqref{eqn:getzler_symbol} and the divisor equation,
we obtain a {\em difference} equation,
\begin{equation*}
f(k+2)-2f(k+1)+f(k)-h(l) \sim 0\, ,
\end{equation*}
from \eqref{eqn:difference_inital}.
A linear algebraic result \cite[Lemma 6.2]{w}
about solutions to 
the difference equation then yields  
\begin{equation}\label{eqn:difference1}
f(i)\sim -\frac{i(l-i)}{2}h(l) \, .
\end{equation}
for $0\leq i \leq l$.

We can write the relation \eqref{eqn:difference1} in the following form.
Let $l\geq 2$, and let $i+j =l$ for $i,j\geq 0$. Then,
\begin{equation}\label{eqn:differenceq}
f(i)\sim -\frac{ij}{2}h(l) \, .
\end{equation}
The case $j=2$ of \eqref{eqn:differenceq} with $b=1$ is
\begin{equation*}
\langle a\star D^{*i}, D^{*2},\phi_1,...,\phi_m\rangle^{\hilbnc}_{1,\, \fbox{\tiny{2+m}}}\sim -i\langle a\star D^{*i}, D^{*2},\phi_1,...,\phi_m\rangle^{\hilbnc}_{1,\, \fbox{\tiny{2+m}}}.      
\end{equation*}
The case $j=1$ of \eqref{eqn:differenceq} with $b=D$ is
\begin{equation*}
\langle a\star D^{*i}, D\star D^{*1},\phi_1,...,\phi_m\rangle^{\hilbnc}_{1,\, \fbox{\tiny{2+m}}}\sim -\frac{i}{2}\langle a\star D\star D^{*(i-1)}, D^{*2},\phi_1,...,\phi_m\rangle^{\hilbnc}_{1,\, \fbox{\tiny{2+m}}}.        
\end{equation*}
These two relations imply 
\begin{equation} \label{lww5}
\langle D^{*i}, D^{*2},\phi_1,...,\phi_m\rangle^{\hilbnc}_{1,\, \fbox{\tiny{2+m}}}\sim 0
\end{equation} for all $i\geq 1$.
The $i=0$ case, $\langle D^{*0}, D^{*2},\phi_1,...,\phi_m\rangle^{\hilbnc}_{1,\, \fbox{\tiny{2+m}}}\sim 0$ follows from the string equation.

As a consequence of \eqref{lww5}, we have $h(l)\sim 0$ for all $l\geq 2$.
Then,  relation \eqref{eqn:differenceq} yields
\begin{equation*}
\langle a\star D^{*i}, b\star D^{*j},\phi_1,...,\phi_m\rangle^{\hilbnc}_{1,\, \fbox{\tiny{2+m}}}\sim 0 \, 
\end{equation*}
for all $i,j\geq 0$ with $i+j\geq 2$. Again the $i+j=0$ and $i+j=1$
cases follow from the string equation. After setting $a=b=1$,
we conclude
\begin{equation} \label{workh}
\langle D^{*i}, D^{*j},\phi_1,...,\phi_m\rangle^{\hilbnc}_{1,\, \fbox{\tiny{2+m}}}\sim 0\, 
\end{equation}
for all $i,j\geq 0$.

Repeated use of 
relation \eqref{workh} shows that $\langle D^{*i}, D^{*j},\phi_1,...,\phi_m\rangle^{\hilbnc}_{1,\, \fbox{\tiny{2+m}}}$ can be expressed in terms of genus $0$ invariants and {\em $1$-point} genus $1$-functions. Moreover, since we use the divisor equation in the argument,  the resulting 
expressions are linear in $1$-point genus $1$ functions {\em and} their $q\frac{d}{dq}$ derivatives. By the genus $0$ reconstruction result in \cite{op}, the genus $0$ series can be written as rational functions in matrix coefficients of $\mathsf{M}^{\Hilb}_D(q)$.  The proof of Theorem \ref{Red1} is complete. \qed

\subsection{Givental's formula in genus 1}\label{sec:recon_multipt_giv}

Consider the (full) genus $1$ Gromov-Witten potential
\begin{equation*}
\mathcal{F}_1^{\hilbnc}(\mathsf{t},q)=\sum_{d\geq 0}\sum_{\ell\geq 0}\frac{q^d}{\ell!}\langle \mathsf{t},...,\mathsf{t}\rangle^{\hilbnc}_{1,\, \fbox{{\tiny $\ell$}}\, 
,d},\quad \mathsf{t}\in H_{\mathsf{T}}^*(\hilbnc)\, .    
\end{equation*}
We have, for $v_1,...,v_k\in H_{\mathsf{T}}^*(\hilbnc)$, 
\begin{equation*}
\langle v_1,...,v_k\rangle_{1}^{\hilbnc}=\nabla_{v_1}\nabla_{v_2}...\nabla_{v_k}\mathcal{F}_1^{\hilbnc}(\mathsf{t},q)\big|_{\mathsf{t}=0}\, . \end{equation*}
Since the quantum cohomology of $\hilbnc$ is semisimple, we can apply Givental's formula \cite{g1}:
\begin{equation*}\label{eqn:giv_g1_formula}
\nabla_v \mathcal{F}_1^{\hilbnc}(\mathsf{t},q)=\frac{1}{2}\sum_{i=1}^{|\mathsf{Part}(n)|} \RRR^1_{ii}\cdot \nabla_{v} u_i+\frac{1}{48} \nabla_v \log \Big(\, \prod_{i=1}^{|\mathsf{Part(n)}|} \Delta_i\Big)\, .    
\end{equation*}
We follow here the notation of \cite{g1}
and refer the reader to 
\cite{g1,LeeP} for an exposition:
\begin{enumerate}
\item[$\bullet$] $\RRR^1_{ii}$ are matrix coefficients of the first term of the classifying 
$\mathsf{R}$-matrix, $$\RRR=\mathsf{Id}+\RRR^1\cdot z+\RRR^2\cdot z^2+... \, ,$$
\item[$\bullet$] $u_1, \ldots, u_{|\mathsf{Part}(n)|}$ are the canonical coordinates,
\item[$\bullet$] $\Delta_1, \ldots, \Delta_{|\mathsf{Part}(n)|}$ are the inverses of the squares of the lengths of the corresponding idempotents $\epsilon_1, \ldots,
\epsilon_{|\mathsf{Part}(n)|}$.
\end{enumerate}
By Theorem \ref{Red1}, we need only consider the 1-point invariants $\langle D^{\star k}\rangle_1^{\hilbnc}$. Givental's formula then can be written as:
\begin{equation}\label{ppssw}
\la D^{\star k} \ra^{\hilbnc}_{1}=\frac{1}{2}\sum_{i=1}^{|\mathsf{Part}(n)|} \RRR^1_{ii}\, \big|_{\mathsf{t}=0}  \cdot \nabla_{D^{\star k}} u_i\, \big|_{\mathsf{t}=0}+\frac{1}{48}\nabla_{D^{\star k}} \log \Big(\, \prod_{i=1}^{|\mathsf{Part(n)|}} \Delta_i\Big)\, \big|_{\mathsf{t}=0}\, .    
\end{equation}


We explain first how to explicitly compute the functions $\Delta_i$. 
Denote by $\{e_1,..., e_{|\mathsf{Part}(n)|}\}$ the distinct eigenvalues \cite{op} of $D\star_{\mathsf{t}}$ and let $\{v_1,..., v_{\mathsf{Part}(n)}\}$ be the corresponding eigenvectors.{\footnote{Here, the
symbol $\star_{\mathsf{t}}$ denotes the big quantum product.}
We have
\begin{equation*}
v_i \star_{\mathsf{t}} v_j =\delta_{ij}c_i v_i \, ,    
\end{equation*}
so the idempotents are 
$\epsilon_i=v_i/c_i$.
By the Frobenius property, 
\begin{equation*}
c_i\langle v_i,1\rangle=\langle v_i\star_{\mathsf{t}} v_i,1\rangle=\langle v_i,v_i\rangle\, .    
\end{equation*}
We then compute:
\begin{equation*}
\Delta_i=\frac{1}{\langle \epsilon_i,\epsilon_i\rangle}=\frac{c_i^2}{\langle v_i,v_i\rangle}=\frac{1}{\langle v_i,v_i\rangle}\cdot \frac{\langle v_i,v_i\rangle^2}{\langle v_i,1\rangle^2}=\frac{\langle v_i,v_i\rangle}{\langle v_i,1\rangle^2}.
\end{equation*}
Since the eigenvectors $v_i$ are found by solving the equation 
$D\star_{\mathsf{t}} v_i=e_i v_i$, the components of $v_i$ are rational functions of $e_i$ with coefficients in the matrix coefficients of $D\star_{\mathsf{t}}$. Therefore,  the 
$\Delta_i$ are rational functions in the eigenvalues $e_i$ with coefficients in the field 
$\mathbb{Q}(t_1,t_2,q)[[\mathsf{t}]]$,
where
 the
matrix coefficients of $D\star_{\mathsf{t}}$ lie. 

To evaluate the term
$\nabla_{D^{\star k}} \log (\prod_i \Delta_i)$,
we need only evaluate symmetric rational functions in the derivatives of the eigenvalues $e_i$ with coefficients in 
$\mathbb{Q}(t_1,t_2,q)[[\mathsf{t}]]$.
By Proposition \ref{prop:sym_function} of Appendix \ref{AAAA},  these expressions lie in the field of rational functions of derivatives of the symmetric functions of $e_i$ with coefficients in
$\mathbb{Q}(t_1,t_2,q)[[\mathsf{t}]]$. The outcome is an explicit calculation
of $$\nabla_{D^{\star k}} \log (\prod_i \Delta_i)\big|_{\mathsf{t}=0} \in  \mathbb{Q}(t_1,t_2,q)\, .$$ 
The same argument can be used to calculate derivatives
$$\nabla_{v_1} \nabla_{v_2} \cdots \nabla_{v_k}\log (\prod_i \Delta_i)\,\big|_{\mathsf{t}=0} \in  \mathbb{Q}(t_1,t_2,q)\, .$$

Next, we consider the term
$\nabla_{D^{\star k}} u_i\, |_{\mathsf{t}=0}$ 
of  \eqref{ppssw}.
Since the eigenvalues of $D^{\star_{\mathsf{t}} k}\star_{\mathsf{t}}$
are simply the $k^{th}$ powers of
the eigenvalues of $D\star_{\mathsf{t}}$, we have
  $$\nabla_{D^{\star k}} u_i\, \big|_{\mathsf{t}=0}=\Big(\nabla_{D} u_i\, \big|_{\mathsf{t}=0}\Big)^k = e_i^k\, \big|_{\mathsf{t}=0}\, .$$ 
Hence, $\nabla_{D^{\star k}} u_i\, |_{\mathsf{t}=0}$ is also determined by the genus 0 theory of $\hilbnc$.

The difficulty in applying formula \eqref{ppssw} to 
calculate  
$\la D^{\star k} \ra^{\hilbnc}_{1}$
lies in controlling the $\mathsf{R}$-matrix terms. We will use our calculation of $\la D^{} \ra^{\hilbnc}_{1}$ to determine
$\mathsf{R}^1_{ii} \, |_{\mathsf{t}=0}$ 
up to the nondegeneracy of the Wronskian.

\subsection{Proof of Theorem \ref{Red2}}\label{sec:pf_Red2}

By the divisor equation,
$$
\la D,...,D\ra^{\hilbnc}_{1,\, \fbox{{\tiny $\ell$}}}=\left(q\frac{d}{dq}\right)^{\ell-1}\la D\ra^{\hilbnc}_{1}\, .
$$ 
We define 
\begin{equation*}
\delta_l=\la D,..,D\ra_{1,\,\fbox{\tiny{$\ell$}}}-\left(q\frac{d}{dq}\right)^{l-1}\frac{1}{48} \nabla_D \log \Big(\, \prod_{i=1}^{|\mathsf{Part(n)}|} \Delta_i\Big) \big|_{\mathsf{t}=0}.
\end{equation*}
By Theorem \ref{Equiv1} and the discussion in Section \ref{sec:recon_multipt_giv}, $\delta_l \in \mathbb{Q}(t_1,t_2,q)$ can be explicitly calculated.

By (\ref{ppssw}), we have, for $\ell\geq 1$,
\begin{equation}\label{eqn:derivative_giv}
\begin{split}
2\delta_\ell&=\left(q\frac{d}{dq}\right)^{\ell-1}\left(\sum_{i=1}^{|\mathsf{Part}(n)|} \RRR^1_{ii}\,  \big|_{\mathsf{t}=0}\cdot \nabla_{D} u_i\, \big|_{t=0}\right)\\
&=\sum_{i=1}^{|\mathsf{Part}(n)|}\sum_{k=0}^{\ell-1}\binom{\ell-1}{k} \left(q
\frac{d}{dq}\right)^k \RRR^1_{ii}\, \big|_{\mathsf{t}=0} \cdot \left(q\frac{d}{dq}\right)^{\ell-1-k} \nabla_D u_i\, \big|_{\mathsf{t}=0}\, .
\end{split}
\end{equation}

By the construction of the $\mathsf{R}$-matrix \cite{g1}, the derivatives $q\frac{d}{dq}\RRR_{ii}^1\, \big|_{\mathsf{t}=0}$ are given by
\begin{equation*}
q\frac{d}{dq}\RRR_{ii}^1\,\big|_{\mathsf{t}=0}=\sum_l \RRR_{il}^1\, \big|_{\mathsf{t}=0}\cdot \left( \nabla_D u_l\, \big|_{\mathsf{t}=0}-\nabla_D u_i\, \big|_{\mathsf{t}=0}\right)\cdot \RRR_{li}^1\,\big|_{\mathsf{t}=0}\, .
\end{equation*}
The off-diagonal terms $\RRR_{il}^1\, \big|_{\mathsf{t}=0}$ are computed by the equation
\begin{equation*}
\Psi^{-1}\cdot q\frac{d}{dq}\Psi=[\nabla_DU, \RRR^1]\, ,    
\end{equation*}
where $\nabla_DU$ is the diagonal matrix with diagonal entries $\nabla_D u_i\, \big|_{\mathsf{t}=0}$ and $\Psi$ is the matrix whose columns are normalized eigenvectors. The functions $q\frac{d}{dq}\RRR_{ii}^1\, \big|_{\mathsf{t}=0}$ are therefore rational functions of the
$q\frac{d}{dq}$ derivatives of the eigenvalues $e_i|_{\mathsf{t}=0}$ with coefficients in the field  $\mathbb{Q}(t_1,t_2,q)$, where
the matrix coefficients of 
$\mathsf{M}^{\hilbnc}_D(q)$ lie.
The higher order derivatives
 $$\left(q\frac{d}{dq}\right)^{k>1} \RRR_{ii}^1\, \big|_{\mathsf{t}=0}$$
 are rational functions of the $q\frac{d}{dq}$ derivatives of the eigenvalues
 $e_i|_{\mathsf{t}=0}$  with coefficients in $\mathbb{Q}(t_1,t_2,q)$.


We would like to calculate the diagonal terms $\RRR_{ii}^1\, \big|_{\mathsf{t}=0}$. Define the column vector 
\begin{equation*}
2\vec{\delta}=(2\delta_1,2\delta_2,...,2\delta_{|\mathsf{Part}(n)|})^T
\end{equation*}
of length $|\mathsf{Part}(n)|$ and the column vector
\begin{multline*}
\vec{r}=\left(\RRR^1_{11}\, \big|_{\mathsf{t}=0},...,\RRR^1_{|\mathsf{Part}(n)||\mathsf{Part}(n)|}\, \big|_{\mathsf{t}=0},
q\frac{d}{dq}\RRR^1_{11}\,\big|_{\mathsf{t}=0},..., q\frac{d}{dq}\RRR^1_{|\mathsf{Part}(n)||\mathsf{Part}(n)|}\, \big|_{\mathsf{t}=0},\right.
...,
\\[3pt]
\left. \big(q\frac{d}{dq}\big)^{|\mathsf{Part}(n)|-1}\RRR^1_{11}\, \big|_{\mathsf{t}=0},..., \big(q\frac{d}{dq}\big)^{|\mathsf{Part}(n)|-1}\RRR^1_{|\mathsf{Part}(n)||\mathsf{Part}(n)|}\, \big|_{\mathsf{t}=0}\right)^T 
\end{multline*}
of length $|\mathsf{Part}(n)|^2$.
Equation \eqref{eqn:derivative_giv} can be written as a matrix equation
\begin{equation}\label{eqn:der_giv_system1}
2\vec{\delta}=\mathsf{A}\, \vec{r}\, .
\end{equation}
Here, $\mathsf{A}$ is a $|\mathsf{Part}(n)|\times |\mathsf{Part}(n)|^2$ matrix of the shape
\begin{equation*}
\mathsf{A}=(\mathsf{W}\,|\,\widehat{\mathsf{W}}),   \end{equation*}
where $\mathsf{W}$ is the  $|\mathsf{Part}(n)|\times |\mathsf{Part}(n)|$ matrix given by
\begin{equation*}
\mathsf{W}=\left(\begin{array}{cccc}
\nabla_D u_1\, \big |_{\mathsf{t}=0} & \nabla_D u_2\, \big|_{\mathsf{t}=0} &\ldots & \nabla_D u_{|\mathsf{Part}(n)|}\, |_{\mathsf{t}=0}\\[10pt]
q\frac{d}{dq} \nabla_D u_1\, \big|_{\mathsf{t}=0} & q\frac{d}{dq} \nabla_D u_2\, \big|_{\mathsf{t}=0} &\ldots & q\frac{d}{dq} \nabla_D u_{|\mathsf{Part}(n)|}\, \big|_{\mathsf{t}=0}\\[10pt]\vdots& \vdots& \vdots& \vdots\\[10pt](q\frac{d}{dq})^{|\mathsf{Part}(n)|-1}\nabla_D u_1\, \big|_{\mathsf{t}=0} & (q\frac{d}{dq})^{|\mathsf{Part}(n)|-1}\nabla_D u_2\, \big|_{\mathsf{t}=0} &\ldots & (q\frac{d}{dq})^{|\mathsf{Part}(n)|-1} \nabla_D u_{|\mathsf{Part}(n)|} \, \big|_{\mathsf{t}=0}
\end{array}\right)
\end{equation*}
and 
$\widehat{\mathsf{W}}$
is an explicit
$|\mathsf{Part}(n)| \times (|\mathsf{Part}(n)|^2-|\mathsf{Part}(n)|)$ matrix with entries given by 
$$\binom{\ell-1}{k} \left(q\frac{d}{dq}\right)^{\ell-1-k} \nabla_D u_i\, \big|_{\mathsf{t}=0}\,$$
for $k>0$.
Since 
$\nabla_D u_i\, \big|_{\mathsf{t}=0}=e_i \, \big|_{\mathsf{t}=0}$,
the matrix coefficients of $\mathsf{A}$ are 
$q\frac{d}{dq}$
derivatives of the eigenvalues $e_i\,\big|_{\mathsf{t}=0}$.

The matrix
$\mathsf{W}$ is the {\em Wronskian} of the eigenvalues $e_i\, |_{\mathsf{t}=0}$ as functions of $\log(q)$. Since the eigenvalues are analytic functions of $\log(q)$, a result of Boecher (see \cite[Lemma 1.12]{vPS}) implies that the eigenvalues are linearly independent over $\mathbb{Q}(t_1,t_2)$ if and only if  $\det(\mathsf{W})$ is not identically $0$.

If $\det(\mathsf{W})$ is nonzero, then $\mathsf{W}$ is invertible.  Equation (\ref{eqn:der_giv_system1}) then implies
\begin{equation}\label{eqn:der_giv_system2}
\mathsf{W}^{-1}\, 2\vec{\delta}=\mathsf{W}^{-1} \mathsf{A}\,  \vec{r}.    
\end{equation}
The $|\mathsf{Part}(n)|\times |\mathsf{Part}(n)|^2$ matrix $\mathsf{W}^{-1} \mathsf{A}$ is of the form
\begin{equation*}
\mathsf{W}^{-1} \mathsf{A}=(\mathsf{I}_{|\mathsf{Part}(n)|\times |\mathsf{Part}(n)|}\, |\, \mathsf{W}^{-1} \widehat{\mathsf{W}})\, .
\end{equation*}
Via the identity matrix $\mathsf{I}_{|\mathsf{Part}(n)|\times |\mathsf{Part}(n)|}$, equation
\eqref{eqn:der_giv_system2}
yields equations for
 $$\RRR^1_{11}\,  \big|_{\mathsf{t}=0}\, ,\, ...\, ,\, \RRR^1_{|\mathsf{Part}(n)||\mathsf{Part}(n)|}\, \big|_{\mathsf{t}=0}$$ 
 in terms of the functions $\delta_1,...,\delta_{|\mathsf{Part}(n)|}$, the
 $q\frac{d}{dq}$ derivatives of the
 eigenvalues $e_i \, |_{\mathsf{t}=0}$, and the higher $q\frac{d}{dq}$ derivatives of the functions $\mathsf{R}^1_{ii}\, |_{\mathsf{t}=0}$.
 
 We have proven  that $\RRR^1_{11}\, \big|_{\mathsf{t}=0}\, , ... ,\, \RRR^1_{|\mathsf{Part}(n)||\mathsf{Part}(n)|}\, \big|_{\mathsf{t}=0}$ can be computed explicitly from 
\begin{equation*}\label{eqn:inv_div_insertion}
\Big\{\, \la D,...,D\ra^{\hilbnc}
_{1,\,\fbox{{\tiny $\ell$}}}
=\left(q\frac{d}{dq}\right)^{\ell-1}\la D\ra^{\hilbnc}_{1}\, 
\Big\}_{k=1}^{|\mathsf{Part}(n)|}
\end{equation*}
and rational functions of the $q\frac{d}{dq}$ derivatives of the  eigenvalues $e_i \, \big|_{\mathsf{t}=0}$ with coefficients in $\mathbb{Q}(t_1,t_2,q)$. \qed



\subsection{Proof of Theorem \ref{Red3}}
By construction,
the final expressions for
$$\RRR^1_{11}\, \big|_{\mathsf{t}=0}\, , \, ...\, ,\, \RRR^1_{|\mathsf{Part}(n)||\mathsf{Part}(n)|}\, \big|_{\mathsf{t}=0}$$
in terms of 
the 
$q\frac{d}{dq}$ derivatives of the  eigenvalues 
$$e_1 \, \big|_{\mathsf{t}=0}\, ,\, 
\ldots\, ,\,  e_{\mathsf{Part}(n)}\, \big|_{\mathsf{t}=0}$$ with coefficients in $\mathbb{Q}(t_1,t_2,q)$
are equivariant under
permutations of the
indices. Therefore,
after substitution in 
\begin{equation*}
\la D^{\star k} \ra^{\hilbnc}_{1}=\frac{1}{2}\sum_{i=1}^{|\mathsf{Part}(n)|} \RRR^1_{ii}\, \big|_{\mathsf{t}=0}  \cdot \nabla_{D^{\star k}} u_i\, \big|_{\mathsf{t}=0}+\frac{1}{48}\nabla_{D^{\star k}} \log \Big(\, \prod_{i=1}^{|\mathsf{Part(n)|}} \Delta_i\Big)\, \big|_{\mathsf{t}=0}\, ,
\end{equation*}
the series
$\la D^{\star k} \ra^{\hilbnc}_{1}$
are {\em symmetric}
rational functions of the $q\frac{d}{dq}$ derivatives of the  eigenvalues $e_i \, \big|_{\mathsf{t}=0}$ with coefficients in $\mathbb{Q}(t_1,t_2,q)$. After an application of Proposition \ref{prop:sym_function} of Appendix \ref{AAAA}, the series $\la D^{\star k} \ra^{\hilbnc}_{1}$
can be effectively reconstructed from $\blang D \brang_{1}^{\hilbnc}$ and $\mathsf{M}_D^{\hilbnc}(q)$. \qed

\subsection{The Wronskian}

We formulate the following nondegeneracy conjecture for the quantum cohomology of $\hilbnc$.

\begin{conjecture}\label{conj:wronskian}
For all $n\geq 1$, the Wronskian matrix $\mathsf{W}$ associated
to $\hilbnc$ is nondegenerate:
\begin{equation*}
\det(\mathsf{W})\neq 0.    
\end{equation*}
\end{conjecture}

We have verified
Conjecture \ref{conj:wronskian}
for $n\leq 7$ by computer
calculations. As discussed in Section \ref{sec:pf_Red2}, Conjecture \ref{conj:wronskian}
can reformulated as the assertion that the
eigenvalues 
$$e_1 \, \big|_{\mathsf{t}=0}\, ,\, 
\ldots\, ,\,  e_{\mathsf{Part}(n)}\, \big|_{\mathsf{t}=0}$$ 
are {\em linearly independent} over
$\mathbb{Q}(t_1,t_2)$.

\section{Calculations}

\noindent $\bullet$
For all $n$, the series $\la (1^n)\ra_1^{\hilbnc}$ has only a constant term in $q$. The calculation was already discussed in Section \ref{sec:modularity}:
\begin{equation} \label{kked}
 \la (1^n)\ra_1^{\mathsf{Hilb}^{n}(\mathbb{C}^2)} = 
 \text{Coeff}_{Q^n}\Big[
 -\frac{1}{24}
\frac{t_1+t_2}{t_1t_2}\cdot \mathcal{P}(Q) \log \mathcal{P}(Q)\Big]\, .
\end{equation}

\vspace{8pt}
\noindent $\bullet$
For all $n$, the series $\la (2, 1^{n-2})\ra_1^{\hilbnc}=-\la \, D\, \ra_1^{\hilbnc}$ is evaluated by Theorem \ref{Equiv1},

\begin{equation} \label{lled}
\la (2,1^{n-2})  \ra_{1}^{\mathsf{Hilb}^n(\mathbb{C}^2)} = 
\frac{1}{24} \frac{(t_1+t_2)^2}{t_1t_2}
\left( \mathsf{Tr}_n + \sum_{k=2}^{n-1} {\sigma_{-1}(n-k)} \mathsf{Tr}_k \right) \, .
\end{equation}

For $2\leq n \leq 5$ and every partition 
$\mu$ of $n$,
we present here closed formulas for the 1-point series 
$$\la\, \mu\,\ra_1^{\hilbnc}\in \mathbb{Q}(t_1,t_2,q)\, .$$
We use a combination of inputs: the full genus 0 theory, the evaluations \eqref{kked} and \eqref{lled}, and Getzler's equation. Once the
1-point series are known, Theorem \ref{Red1} effectively determines the full genus 1 Gromov-Witten theory of $\hilbnc$.

While Givental's formula was used in Section \ref{sec:pf_Red2} to prove a structural reconstruction result, calculations are more efficiently obtained from the known series by Getzler's equation. 

\vspace{8pt}
\noindent $\bullet$ For $n=2$, we have:
\begin{equation*}
\begin{split}
\la (1,1)\ra_{1}^{\mathsf{Hilb}^2(\mathbb{C}^2)}&=-\frac{1}{24}\cdot \frac{t_1+t_2}{t_1t_2}\cdot \frac{5}{2}\, ,\\[5pt]    
\langle \, (2)\,  \rangle_{1}^{\mathsf{Hilb}^2(\mathbb{C}^2)} &=
\frac{1}{24}\frac{(t_1+t_2)^2}{t_1t_2} \cdot \frac{q+1}{q-1}\, .
\end{split}
\end{equation*}

\noindent $\bullet$ For $n=3$, we have:

{\tiny
\begin{equation*}
\begin{split}
\la (1,1,1)\ra_{1}^{\mathsf{Hilb}^3(\mathbb{C}^2)}&=-\frac{1}{24}\cdot \frac{t_1+t_2}{t_1t_2}\cdot\frac{29}{6}\, ,\\[5pt]
\langle \, (2,1)\,  \rangle_{1}^{\mathsf{Hilb}^3(\mathbb{C}^2)}
&=\frac{1}{24}\cdot \frac{(t_1+t_2)^2}{t_1t_2} \cdot \frac{5q^3-3q^2-3q+5}{(q-1)(q^2-q+1)}\, ,\\[5pt]
\la (3)\ra_{1}^{\mathsf{Hilb}^3(\mathbb{C}^2)}&= \frac{-1}{12}\frac{(t_1+t_2)}{t_1t_2}\frac{(t_1^2+\frac{1}{3}t_1t_2+t_2^2)(q^4+1)-\frac{1}{2}(t_1^2-\frac{17}{3}t_1t_2+t_2^2)(q^3+q)-(3t_1^2+13t_1t_2+3t_2^2)q^2}{(q^2-q+1)^2}\, .    
\end{split}
\end{equation*}
}

\vspace{8pt}
\noindent $\bullet$ For $n=4$, we have:

{\tiny
\begin{align*}
\la (1,1,1,1)\ra_{1}^{\mathsf{Hilb}^4(\mathbb{C}^2)}=&\,-\frac{1}{24}\cdot \frac{t_1+t_2}{t_1t_2}\cdot\frac{109}{12}\, ,\\[5pt]
\langle \, (2,1,1)\,  \rangle_{1}^{\mathsf{Hilb}^4(\mathbb{C}^2)}=&\,\frac{1}{24}\cdot\frac{(t_1+t_2)^2}{t_1t_2}\cdot\frac{35q^5-28q^4+23q^3+23q^2-28q+35}{2(q-1)(q^2+1)(q^2-q+1)}\, ,\\[5pt]
\hspace{-20pt} \langle (2,2)\rangle_{1}^{\mathsf{Hilb}^4(\mathbb{C}^2)}
=&\,\frac{-(t_1+t_2)}{t_1t_2}\cdot  
\frac{3}{16(q^2+1)^2(q-1)^2} \cdot
\\
&\ \left((t_1^2+\frac{1}{18}t_1t_2+t_2^2)(q^6+1)+(\frac{2}{9}t_1^2+5t_1t_2+\frac{2}{9}t_2^2)(q^5+q) \right. \\
&
\left.+(\frac{11}{9}t_1^2-\frac{77}{18}t_1t_2+\frac{11}{9}t_2^2)(q^4+q^2)+(4t_1^2+\frac{170}{9}t_1t_2+4t_2^2)q^3\right)\, ,\\[50pt]
\langle (3,1)\rangle_{1}^{\mathsf{Hilb}^4(\mathbb{C}^2)}
=&\,\frac{-(t_1+t_2)}{t_1t_2}\cdot\frac{1}{2(q^2-q+1)^2(q^2+1)^2}\cdot\\[5pt]
&\left((t_1^2+\frac{23}{36}t_1t_2+t_2^2)(q^8+1)-(\frac{5}{6}t_1^2-\frac{41}{36}t_1t_2+\frac{5}{6}t_2^2)(q^7+q)+(\frac{1}{3}t_1^2-\frac{257}{36}t_1t_2+\frac{1}{3}t_2^2)(q^6+q^2)\right.\\
&\left.+(\frac{5}{6}t_1^2+\frac{137}{12}t_1t_2+\frac{5}{6}t_2^2)(q^5+q^3)-(\frac{8}{3}t_1^2+\frac{170}{6}t_1t_2+\frac{8}{3}t_2^2)q^4\right)\, ,\\[5pt]
\langle (4)\rangle_{1}^{\mathsf{Hilb}^4(\mathbb{C}^2)}
=&\, \frac{(t_1+t_2)^2}{t_1t_2}\cdot\frac{(q+1)}{4(q^2-q+1)(q^2+1)^3(q-1)}\cdot\\[5pt]
&\left((t_1^2-\frac{5}{6}t_1t_2+t_2^2)(q^8+1)-(\frac{5}{3}t_1^2-\frac{25}{6}t_1t_2 +\frac{5}{3}t_2^2)(q^7+q)+(2t_1^2-13t_1t_2+2t_2^2)(q^6+q^2)\right.\\[5pt]
&\left.+(3t_1^2+\frac{69}{2}t_1t_2+3t_2^2)(q^5+q^3)-(\frac{10}{3}t_1^2+39t_1t_2+\frac{10}{3}t_2^2)q^4\right)\, .
\end{align*}
}

\vspace{8pt}
\noindent $\bullet$ For $n=5$, we have:
\begin{equation*}
\la 1\ra_{1}^{\mathsf{Hilb}^5(\mathbb{C}^2)}=-\frac{1}{24}\cdot \frac{t_1+t_2}{t_1t_2}\cdot \frac{907}{60}\,.    
\end{equation*}

\noindent We will write remaining series in terms of the traces{\footnote{The subscript $m$ of $\mathsf{Tr}_m^{\mu}$ is redundant since $m=|\mu|$, but is included for clarity.}} of quantum multiplication,
\begin{equation*}
\mathsf{Tr}_m^{\mu}=\mathsf{trace}\left(\,  |\mu\rangle\, \star : QH_\T^*(\mathsf{Hilb}^m(\mathbb{C}^2))\to QH_\T^*(\mathsf{Hilb}^m(\mathbb{C}^2))\, \right)\, .  
\end{equation*}
The trace which appears in Theorem \ref{Equiv1} can be written as
$$\mathsf{Tr}_m = -\frac{1}{t_1+t_2} \mathsf{Tr}_m^{(2,1^{m-2})}\, .$$
Then, we have:

{\tiny
\begin{align*}
\langle (2,1,1,1)\rangle_{1}^{\mathsf{Hilb}^5(\mathbb{C}^2)}=&\, -\frac{t_1+t_2}{18t_1t_2}\mathsf{Tr}_2^{(2)}-\frac{t_1+t_2}{16t_1t_2}\mathsf{Tr}_3^{(2,1)}-\frac{t_1+t_2}{24t_1t_2}\mathsf{Tr}_4^{(2,1,1)}-\frac{t_1+t_2}{24t_1t_2}\mathsf{Tr}_5^{(2,1,1,1)}\\
=&\, \frac{1}{24}\cdot \frac{(t_1+t_2)^2}{t_1t_2} \cdot 
{{\left(\frac{272q^9-539q^8+760q^7-629q^6+302q^5+302q^4-629q^3+760q^2-539q+272}{6(q-1)(q^2+1)(q^2-q+1)(q^4-q^3+q^2-q+1)}\right)}}\, ,\\[5pt]
%
\langle (2,2,1)\rangle_{1}^{\mathsf{Hilb}^5(\mathbb{C}^2)}=&\,\frac{775t_1^2+733t_1t_2+775t_2^2}{1200(t_1+t_2)}+\frac{1}{200(t_1+t_2)}\mathsf{Tr}_3^{(3)}+\frac{-10t_1^2-13t_1t_2-10t_2^2}{240t_1t_2(t_1+t_2)}\mathsf{Tr}_4^{(2,2)}\\[5pt]
&+\frac{-3}{200(t_1+t_2)}\mathsf{Tr}_4^{(3,1)}+\frac{-25t_1^2-68t_1t_2-25t_2^2}{600t_1t_2(t_1+t_2)}\mathsf{Tr}_5^{(2,2,1)}+\frac{1}{200(t_1+t_2)}\mathsf{Tr}_5^{(3,1,1)}\\[5pt]
&+\frac{50t_1^2-139t_1t_2+50t_2^2}{2400t_1t_2(t_1+t_2)}\mathsf{Tr}_2^{(2)}\cdot \mathsf{Tr}_2^{(2)}+\frac{-25t_1^2-44t_1t_2-25t_2^2}{600t_1t_2(t_1+t_2)}\mathsf{Tr}_2^{(2)}\cdot \mathsf{Tr}_3^{(2,1)}+\frac{1}{100(t_1+t_2)}\mathsf{Tr}_2^{(2)}\cdot \mathsf{Tr}_5^{(2,1,1,1)}\, ,\\[5pt]
%
\langle (3,1,1)\rangle_{1}^{\mathsf{Hilb}^5(\mathbb{C}^2)}=&\, \frac{175t_1^2+77t_1t_2+175t_2^2}{300(t_1+t_2)}+\frac{-225t_1^2-532t_1t_2-225t_2^2}{3600t_1t_2(t_1+t_2)}\mathsf{Tr}_3^{(3)}+\frac{1}{120(t_1+t_2)}\mathsf{Tr}_4^{(2,2)}+\frac{-25t_1^2-59t_1t_2-25t_2^2}{600t_1t_2(t_1+t_2)}\mathsf{Tr}_4^{(3,1)}\\[5pt]
&+\frac{-3}{100(t_1+t_2)}\mathsf{Tr}_5^{(2,2,1)}+\frac{-25t_1^2-47t_1t_2-25t_2^2}{600t_1t_2(t_1+t_2)}\mathsf{Tr}_5^{(3,1,1)}+\frac{-2}{75(t_1+t_2)}\mathsf{Tr}_2^{(2)}\cdot \mathsf{Tr}_2^{(2)}\\[5pt]
&+\frac{1}{100(t_1+t_2)}\mathsf{Tr}_2^{(2)}\cdot \mathsf{Tr}_3^{(2,1)}+\frac{1}{100(t_1+t_2)}\mathsf{Tr}_2^{(2)}\cdot \mathsf{Tr}_5^{(2,1,1,1)}\, ,\\[25pt]
%
\langle (3,2)\rangle_{1}^{\mathsf{Hilb}^5(\mathbb{C}^2)}
=&\,\frac{865t_1^2+1556t_1t_2+865t_2^2}{1800(t_1+t_2)}\mathsf{Tr}_2^{(2)}+\frac{-158t_1^2-91t_1t_2-158t_2^2}{900(t_1+t_2)}\mathsf{Tr}_3^{(2,1)}+\frac{20t_1^2+79t_1t_2+20t_2^2}{450(t_1+t_2)}\mathsf{Tr}_4^{(2,1,1)}+\frac{-1}{300(t_1+t_2)}\mathsf{Tr}_4^{(4)}\\[5pt]
&+\frac{35t_1^2-77t_1t_2+35t_2^2}{900(t_1+t_2)}\mathsf{Tr}_5^{(2,1,1,1)}+\frac{-25t_1^2-77t_1t_2-25t_2^2}{600t_1t_2(t_1+t_2)}\mathsf{Tr}_5^{(3,2)}+\frac{1}{300(t_1+t_2)}\mathsf{Tr}_5^{(4,1)}+\frac{-25t_1^2-39t_1t_2-25t_2^2}{600t_1t_2(t_1+t_2)}\mathsf{Tr}_2^{(2)}\cdot \mathsf{Tr}_3^{(3)}\\[5pt]
&+\frac{-23}{1800(t_1+t_2)}\mathsf{Tr}_2^{(2)}\cdot \mathsf{Tr}_4^{(3,1)}+\frac{4}{225(t_1+t_2)}\mathsf{Tr}_2^{(2)}\cdot \mathsf{Tr}_4^{(2,2)}+\frac{-4}{225(t_1+t_2)}\mathsf{Tr}_2^{(2)}\cdot \mathsf{Tr}_5^{(2,2,1)}+\frac{23}{1800(t_1+t_2)}\mathsf{Tr}_2^{(2)}\cdot \mathsf{Tr}_5^{(3,1,1)}\, ,\\[5pt]
\langle (4,1)\rangle_{1}^{\mathsf{Hilb}^5(\mathbb{C}^2)}
=&\, \frac{-430t_1^2+743t_1t_2-430t_2^2}{2400(t_1+t_2)}\mathsf{Tr}_2^{(2)}+\frac{-461t_1^2+503t_1t_2-461t_2^2}{4800(t_1+t_2)}\mathsf{Tr}_3^{(2,1)}+\frac{925t_1^2+1771t_1t_2+925t_2^2}{4800(t_1+t_2)}\mathsf{Tr}_4^{(2,1,1)}\\[5pt]
&+\frac{-100t_1^2-223t_1t_2-100t_2^2}{2400t_1t_2(t_1+t_2)}\mathsf{Tr}_4^{(4)}+\frac{-113t_1t_2}{600(t_1+t_2)}\mathsf{Tr}_5^{(2,1,1,1)}+\frac{-1}{50(t_1+t_2)}\mathsf{Tr}_5^{(3,2)}+\frac{-25t_1^2-63t_1t_2-25t_2^2}{600t_1t_2(t_1+t_2)}\mathsf{Tr}_5^{(4,1)}\\[5pt]
&+\frac{1}{600(t_1+t_2)}\mathsf{Tr}_2^{(2)}\cdot \mathsf{Tr}_3^{(3)}+\frac{-11}{600(t_1+t_2)}\mathsf{Tr}_2^{(2)}\cdot \mathsf{Tr}_4^{(3,1)}+\frac{7}{300(t_1+t_2)}\mathsf{Tr}_2^{(2)}\cdot \mathsf{Tr}_4^{(2,2)}\\[5pt]
&+\frac{-7}{300(t_1+t_2)}\mathsf{Tr}_2^{(2)}\cdot \mathsf{Tr}_5^{(2,2,1)}+\frac{11}{600(t_1+t_2)}\mathsf{Tr}_2^{(2)}\cdot \mathsf{Tr}_5^{(3,1,1)}\, ,\\[5pt]
%
\langle (5)\rangle_{1}^{\mathsf{Hilb}^5(\mathbb{C}^2)}
=&\,\frac{t_1t_2(110t_1^2-1391t_1t_2+110t_2^2)}{600(t_1+t_2)}+\frac{-75t_1^2-53t_1t_2-75t_2^2}{300(t_1+t_2)}\mathsf{Tr}_3^{(3)}+\frac{-50t_1^2+77t_1t_2-50t_2^2}{600(t_1+t_2)}\mathsf{Tr}_4^{(2,2)}\\[5pt]
&+\frac{-25t_1^2-59t_1t_2-25t_2^2}{300(t_1+t_2)}\mathsf{Tr}_5^{(2,2,1)}+\frac{75t_1^2+53t_1t_2+75t_2^2}{600(t_1+t_2)}\mathsf{Tr}_5^{(3,1,1)}+\frac{-5t_1^2-16t_1t_2-5t_2^2}{120t_1t_2(t_1+t_2)}\mathsf{Tr}_5^{(5)}\\[5pt]
&+\frac{175t_1^2+94t_1t_2+175t_2^2}{150(t_1+t_2)}\mathsf{Tr}_2^{(2)}\cdot \mathsf{Tr}_2^{(2)}+\frac{44t_1^2+45t_1t_2+44t_2^2}{120(t_1+t_2)}\mathsf{Tr}_2^{(2)}\cdot \mathsf{Tr}_3^{(2,1)}+\frac{-60t_1^2-53t_1t_2-60t_2^2}{200(t_1+t_2)}\mathsf{Tr}_2^{(2)}\cdot \mathsf{Tr}_4^{(2,1,1)}\\[5pt]
&+\frac{15t_1^2+23t_1t_2+15t_2^2}{300(t_1+t_2)}\mathsf{Tr}_2^{(2)}\cdot \mathsf{Tr}_5^{(2,1,1,1)}\, .
\end{align*}
}

\vspace{8pt}
\noindent $\bullet$ There is no obstruction (apart from expected nondegeneracies) to extending the above tables to higher $n$. 
Whether further structures can be found in the $1$-points series  $\la\, \mu\,\ra_1^{\hilbnc}$ is an interesting open question. 
\appendix

\section{On symmetric functions}\label{AAAA}
Let $\mathbb{K}$ be a field of
characteristic 0.
Let $\mathbf{z}=(z_1,\ldots,z_m)$ be vector of $m$ variables. 
Let $f_1(\mathbf{z}), \ldots, f_n(\mathbf{z})$ be $n$ abstract functions, and let
\begin{equation*}
    s_k(\mathbf{z}) = (-1)^k\sum_{I \subseteq [n], |I|=k} \prod_{i \in I} f_i
\end{equation*}
be the elementary symmetric polynomials in $f_1,...,f_n$. 
The {\em discriminant} is defined by  
\begin{equation*}
    \Delta(\mathbf{z}) = \prod_{i \neq j} (f_i-f_j)\, .
\end{equation*}
We consider the following algebras:

\vspace{8pt}
\noindent
$\bullet$  The standard algebra of symmetric functions,
$$\mathsf{Sym} = \mathbb K[f_1, \ldots , f_n]^{S_n} = \mathbb K[s_1, \ldots , s_n]\, ,$$
is defined by taking invariants of the $S_n$-action on $\mathbb K[f_1, \ldots , f_n]$. The $S_n$-action is defined
by permuting the indices of $f_i$.
By construction, $\Delta \in \mathsf{Sym}$.

\vspace{8pt}
\noindent $\bullet$ 
Let $\mathsf{Df}=\big \{ \frac{\partial^{a_1+\ldots +a_m}}{\partial z_1^{a_1}\cdots \partial z_m^{a_m}} f_i\big\}_{(a_1,\ldots,a_m)\in (\mathbb{Z}_{\geq 0})^m}$
be the set of all partial derivatives of the functions $f_1,\ldots,f_n$,
$$\mathsf{Df} =
\big\{ f_1, \ldots, f_n, \ldots, \frac{\partial f_i}{\partial z_j}, \ldots, \frac{\partial^2 f_i}{\partial z_{j_1} \partial z_{j_2}},
\ldots  \big\}\, .$$
The algebra $\mathbb K[   \mathsf{Df}        ]$ of polynomials{\footnote{While $\mathsf{Df}$ is an infinite set of functions, only finitely many appear
in any given polynomial.}} in the functions
$\mathsf{Df}$ carries an $S_n$-action defined by permuting{\footnote{$S_n$ does {\em not} act on the variables $z_j$ nor on the operators $\frac{\partial}{\partial z_j}$.}} the indices of $f_i$.
Let $\mathsf{SymDf}$ be  the algebra of $S_n$-invariants:
    $$\mathsf{SymDf} =\mathbb K[   \mathsf{Df}        ]^{S_n}\, . $$

\vspace{8pt}
\noindent $\bullet$ 
Let $\mathsf{Ds}=\big \{ \frac{\partial^{a_1+\ldots +a_m}}{\partial z_1^{a_1}\cdots \partial z_m^{a_m}} s_i\big\}_{(a_1,\ldots,a_m)\in (\mathbb{Z}_{\geq 0})^m}$
be the set of all partial derivatives of the elementary symmetric functions $s_1, \ldots, s_n$.
Let $\mathbb{K}[\mathsf{Ds}]$ be the algebra of  polynomials in the functions $\mathsf{Ds}$.
   

We present a proof of the following result (which is likely known to experts, but we
were unable to find a reference).

\begin{prop}\label{prop:sym_function}
    $
        \mathsf{SymDf} \subseteq \mathbb{K}[\mathsf{Ds}][1/\Delta]
    $ .
\end{prop}

\begin{proof}
Let ${\mathbf{a}}=(a_1,\ldots, b_m)$
and ${\mathbf{b}}=(b_1,\ldots,b_m)$
both be elements of $(\mathbb{Z}_{\geq 0})^m$. We define
\begin{enumerate}
\item[(i)] $\mathbf{a}\leq \mathbf{b}$ if $a_j \leq b_j$
for all $1\leq j \leq m$,
\item [(ii)] $\mathbf{a}<\mathbf{b}$ if
$\mathbf{a}\leq \mathbf{b}$ and $a_j < b_j$
for some $j$.
\end{enumerate}
Let $\mathsf{Ds}_{\leq \mathsf{b}}=\big \{ \frac{\partial^{a_1+\ldots +a_m}}{\partial z_1^{a_1}\cdots \partial z_m^{a_m}} s_i\big\}_{\mathsf{a}\leq \mathsf{b}}$
be a finite set of 
partial derivatives of the elementary symmetric functions $s_1, \ldots, s_n$.
Similarly, let 
$\mathsf{DY}_{< \mathsf{b}}=\big \{ \frac{\partial^{a_1+\ldots +a_m}}{\partial z_1^{a_1}\cdots \partial z_m^{a_m}} Y\big\}_{\mathsf{a}< \mathsf{b}}$
    be the finite set of partial
    derivatives of a single abstract function
    $Y(\mathbf{z})$.

    Consider the polynomial
    \begin{equation*} 
    P(x) = x^n + s_1 x^{n-1} + \ldots + s_n = \prod_{i}(x-f_i)\, .
    \end{equation*}
    We can take the  $\frac{\partial}{\partial z_j}$
    derivative of the relation $P(f_i) = 0$ :
    \begin{equation}\label{vfvf12}
    \frac{\partial f_i}{\partial z_j} P_x(f_i) + \frac{\partial s_1}{\partial z_j} f_i^{n-1} + \ldots+\frac{\partial s_{n-1}}{\partial z_j}f_i +\frac{\partial s_n}{\partial z_j}=0\, ,
    \end{equation}
    where $P_x$ is the derivative of $P$ as a polynomial in $x$. By a simple calculation, 
\begin{equation*}\label{eqn:discriminant_derivative}
        \Delta = \prod_{i=1}^n P_x(f_i)\, .
    \end{equation*}

Let $\mathbf{b}\in (\mathbb{Z}_{\geq 0})^m$.
By repeatedly taking derivatives
of \eqref{vfvf12}, we find that there is a universal polynomial
    \begin{equation*}
    \Phi_\mathbf{b}(\mathsf{Ds}_{\leq \mathbf{b}}, \mathsf{DY}_{<\mathbf{b}})
\in \mathbb{K}[\mathsf{Ds}_{\leq \mathbf{b}}, \mathsf{DY}_{<\mathbf{b}}]
    \end{equation*}
 which    satisfies the following property: for every $f_i$,
    \begin{equation} \label{pvii}
    \frac{\partial^{b_1+\ldots +b_m} f_i}{\partial z_1^{b_1}\cdots \partial z_m^{b_m}} 
    P_x(f_i) = \Phi_{\mathbf{b}}|_{Y=f_i}\, .
    \end{equation}
    Therefore, there is  a universal polynomial
    $
    \Omega_{\mathbf{b}}(\mathsf{Ds}_{\leq \mathbf{b}}, Y) \in \mathbb K[\mathsf{Ds}_{\leq \mathbf{b}},Y]
    $
    for which
    \begin{equation*}
     \frac{\partial^{b_1+\ldots +b_m} f_i}{\partial z_1^{b_1}\cdots \partial z_m^{b_m}} 
     P_x(f_i)^{N_{\mathbf{b}}} = \Omega_{\mathbf{b}}(
     \mathsf{Ds}_{\leq \mathbf{b}},
     f_i)\, 
    \end{equation*}
    for a (possibly large) integer $N_{\mathbf{b}}$. 
    
    We now take an arbitrary monomial in the functions of $\mathsf{Df}$:
$$M = \prod_{i=1}^n \prod_{u=1}^{v_i} 
\partial_{\mathbf{b}(i,u)} f_i\, ,$$
where we have used the notation
$$ \partial_{\mathbf{b}(i,u)} f_i =
\frac{\partial^{b_1(i,u)+\ldots +b_m(i,u)} f_i}{\partial z_1^{b_1(i,u)}\cdots \partial z_m^{b_m(i,u)}} \,.$$
By \eqref{pvii}, we have
\begin{equation} \label{mwwx}
\Delta^{\sum_{i=1}^n\sum_{u=1}^{v_i} N_{\mathbf{b}(i,u)}} \cdot M =
\prod_{i=1}^n \prod_{u=1}^{v_i} 
\Omega_{\mathbf{b}(i,u)}(\mathsf{Ds}_{\leq \mathbf{b}(i,u)},f_i) \cdot 
\Delta_i^{N_{\mathbf{b}(i,u)}}       \, ,
\end{equation}
where
$\Delta_i= \Delta / P_x(f_i)$.

Consider next the $S_n$-invariant element
$$\mathsf{sym} M= \sum_{\sigma\in S_n}\sigma(M)\in \mathsf{SymDf}\, .$$
Using \eqref{mwwx}, we obtain
\begin{equation}\label{mwwy}
\mathsf{sym} M= \Delta^{-\sum_{i=1}^n\sum_{u=1}^{v_i} N_{\mathbf{b}(i,u)}}
\sum_{\sigma\in S_n}\sigma\left(\prod_{i=1}^n \prod_{u=1}^{v_i} 
\Omega_{\mathbf{b}(i,u)}(\mathsf{Ds}_{\leq \mathbf{b}(i,u)},f_i) \cdot 
\Delta_i^{N_{\mathbf{b}(i,u)}}\right)\, 
\end{equation}
where the right side of \eqref{mwwy} lies in
$\mathbb{K}[\mathsf{Ds}][1/\Delta]$. 
As $M$ varies over all monomials in $\mathbb{K}[\mathsf{Df}]$, the elements 
$\mathsf{sym}M$ generate $\mathsf{SymDf}$. 
\end{proof}

\end{document}